\documentclass{puthesis}

\usepackage{graphicx}
\usepackage{bm}
\usepackage{url}
\usepackage[all]{xy}
\usepackage[dvips]{epsfig}
\usepackage[margin=1.3in]{geometry}
\usepackage{enumitem}
\usepackage{amsfonts, amssymb, amsmath, amsthm, amscd}

\newcommand{\ul}[1]{\underline{#1}}
\newcommand{\wt}[1]{\widetilde{#1}}

\newcommand{\mc}[1]{\mathcal{#1}}

\newcommand{\opn}[1]{\operatorname{#1}}

\newcommand{\abs}[1] {\lvert #1 \rvert}
\newcommand{\gen}[1] {\left\langle #1 \right\rangle}

\newcommand{\rg}{\mathcal{R}_g}
\newcommand{\CF} {\widehat{\operatorname{CF}}}
\newcommand{\HF} {\widehat{\operatorname{HF}}}
\newcommand{\HFK} {\widehat{\operatorname{HFK}}}

\newcommand{\mfs}{\mathfrak{s}}
\newcommand{\Sym}{\text{Sym}}

\newcommand{\dcob}[1]{\wt{\text{Cob}}_{\text{#1+1}}}
\newcommand{\tbf}[1]{\textbf{#1}}
\newcommand{\zoltan}{Zolt\'an Szab\'o}
\newcommand{\petero}{Peter Ozsv\'ath}

\newcommand{\into}{\hookrightarrow}
\newcommand{\onto}{\twoheadrightarrow}

\renewcommand{\L}{\mc{L}}

\newtheorem{theorem}{Theorem}
\newtheorem{lemma}{Lemma}
\newtheorem{proposition}{Proposition}

\newtheorem{defn}{Definition}
\newtheorem{conj}{Conjecture}

\newtheorem{rem}{Remark}

\def\R{\mathbb{R}}
\def\C{\mathbb{C}}
\def\Z{\mathbb{Z}}
\def\Q{\mathbb{Q}}

\def\T{\mathbb{T}}
\def\HH{\mathcal{H}}
\def\x{\mathbf{x}}

\def\T{\mathbb{T}}
\def\AA{\mathcal{A}}
\def\BB{\mathcal{B}}
\def\GG{\mathcal{G}}
\def\HH{\mathcal{H}}
\def\SS{\mathfrak{S}}
\def\minus{\smallsetminus}
\def\Th{{}^{\text{th}}}

\DeclareMathOperator{\sign}{sign}

\DeclareMathOperator{\gr}{gr} 
\DeclareMathOperator{\rank}{rank}

\author{Sam Jay Lewallen}
\adviser{Zolt\'an Szab\'o}
\title{Floerg\r{a}sbord}
\abstract{In this thesis, we prove several results concerning field-theoretic invariants of knots and 3-manifolds. 

In Chapter 2, for any knot $K$ in a closed, oriented 3-manifold $M$, we use $SU(2)$ representation spaces and the Lagrangian field theory framework of Wehrheim and Woodward to define a new homological knot invariant $\mc{S}(K)$. We then use a result of Ivan Smith to show that when $K$ is a (1,1) knot in $S^3$ (a set of knots which includes torus knots, for example), the rank of $\mc{S}(K)\otimes \C$ agrees with the rank of knot Floer homology, $\widehat{HFK}(K)\otimes \C$, and we conjecture that this holds in general for any knot $K$.

In Chapter 3, we prove a somewhat strange result, giving a purely topological formula for the Jones polynomial of a 2-bridge knot $K\subset S^3$. First, for any lens space $L(p,q)$, we combine the $d$-invariants from Heegaard Floer homology with certain Atiyah-Patodi-Singer/Casson-Gordon $\rho$-invariants to define a function $$I_{p,q}: \Z/p\Z \to \Z$$

Let $K = K(p,q)$ denote the 2-bridge knot in $S^3$ whose double-branched cover is $L(p,q)$, let $\sigma(K)$ denote the knot signature, and let $\mc{O}$ denote the set of relative orientations of $K$, which has cardinality $2^{(\# \text{ of components of } K) - 1}$. Then we prove the following formula for the Jones polynomial $J(K)$: $$i^{-\sigma(K)}q^{3\sigma(K)}J(K)= \sum_{o\in\mathcal{O}}(iq)^{2\sigma(K^{o})} +\left(q^{-1}-q^{1}\right)\sum_{\mfs\in\Z/p\Z}(iq)^{I_{p,q}(\mfs)}$$ (here, $i = \sqrt{-1}$).

In Chapter 4, we present joint work with Adam Levine, concerning Heegaard Floer homology and the orderability of fundamental groups. Namely, we prove that if $\widehat{CF}(M)$ is particularly simple, i.e., $M$ is what we call a ``strong $L$-space,'' then $\pi_1(M)$ is not left-orderable. 

}
\acknowledgements{So many people have contributed their time and effort to helping me complete this thesis, and I am continually in awe of the kindness and generosity of my colleagues, friends, and family. It is a pleasure to thank you.

I am enormously grateful to my advisor \zoltan, for all his patience and encouragement. I have also benefited from the mathematical conversations, hospitality, and often great friendships, of so many mathematicians, including (it is a long list, but since it's sad to say goodbye, I'd like to include as many people as I can): Jon Bloom, Josh Greene, Adam Levine, John Baldwin, Matt Hedden, Ciprian Manolescu, Peter Kronheimer, \petero, Paul Kirk, Chuck Livingston, Ed Witten, Alli Gilmore, Jen Hom, Tye Lidman, Sucharit Sarkar, Danny Ruberman, Cliff Taubes, Peter Teichner, Andr\'as Stipsicz, Chris Woodward, Michael McBreen. I'd like to thank Liam Watson, Will Cavendish, and Cotton Seed for their particularly voluminous support and friendship, mathematical and otherwise. Will was the first to introduce me to true mathematical research, in all of its immersive, obsessive glory; I've always admired your vision and your originality. I'd also like to thank Cotton and Will for our amazing summer of cheese and math. I would like to give very special thanks to Henry Cowles, for being such an important and supportive friend to me for so long.

I'm also very grateful to Jill LeClair for all her help throughout my years as a graduate student, and particularly towards the end, for the role she played in making sure this thesis actually got completed. 

Most thanks of all goes to my wonderful family. Words cannot express how much you all mean to me. Thank you so much.

I would also like to thank MTL for your friendship and support. It's been very important to me. 
}
\dedication{To my parents Myrel and Don,\\ my sister Tanya,\\ and my grandparents Anne and Jack.}

\begin{document}

\chapter{Introduction and statements of main results}

Every smooth, closed, oriented $n$-manifold $M$ admits a Morse function $f:M\to [0,1]$, presenting $M$ as a singular fibration over the interval. In other words, $M$ can be visualized as a path of manifolds of one lower dimension -- the regular level sets of $f$ -- starting and ending with the empty manifold, and constant except at the critical points of $f$. At each critical point, the level set is modified by the attachment of a single $n$-cell. More generally, if the boundary $N = \partial M$ is non-empty, we can take $f(N) = \{0,1\}$, and the cobordism $M$ becomes a path of $n$-cell attachments relating $N = f^{-1}(0)$ to $N' = f^{-1}(1)$. 

One might hope to use this picture to define invariants of $M$, in the following way. First, define an invariant $\Gamma(N)$ of $(n-1)$-manifolds, which ought to be easier, and understand the relationship between $\Gamma(N)$ and $\Gamma(N')$ when $N$ and $N'$ are related by an $n$-cell attachment. The invariant associated to $M$ should then be something like a union of invariants $\Gamma(N_i)$ over all regular level sets of $f$, ``quotiented'' by the relation between $\Gamma(N_i)$ and $\Gamma(N_{i+1})$ induced by $f$. For example, using Van-Kampen's theorem, the fundamental group $\pi_1(M)$ can be expressed in this way.

In the last twenty years, ideas from quantum physics have led to an important set of invariants fitting into this framework, the so-called ``topological (quantum) field theories'' (TQFT's) of Atiyah, Segal, Witten, and others (see \cite{atiyah1988new}, for example). Furthermore, as Bott elegantly describes in \cite{bott1993reflection}, when the dimension $n$ equals 3, the fibers of our Morse function $f$ come equipped with compatible holomorphic (and symplectic) structures, induced from a Riemannian metric on $M$. Therefore in this case we can hope to concoct $\Gamma(N)$ via holomorphic methods, and so obtain a particularly rich set of topological field theories in dimension 3. In this thesis, we present three distinct studies on invariants derived from $3d$ field theories of this nature. We will now give a brief description of each of these results.

To begin, taking the dimension $n=3$ in the above discussion, we are advised to associate an invariant $\Gamma(N)$ to smooth 2-dimensional surfaces $N$, making use of their natural holomorphic structures. Insights from theoretical physics have led to two particularly prevalent and powerful examples, each giving an invariant of surfaces with values in the set of holomorphic (alternatively, symplectic) manifolds of arbitrary dimension.  In the first case, we take $$\Gamma_\text{Don}(N) := \mc{M}_r(N), $$ where $\mc{M}_r(N)$ denotes a certain space of holomorphic bundles of rank $r$ on $N$, with respect to some choice of Riemann surface structure. Furthermore, the famous theorem of Narasimhan and Seshadri \cite{narasimhan1965stable} shows that the complex manifold $\mc{M}_r(N)$ is homeomorphic to a symplectic manifold consisting of conjugacy classes of $U(r)$-representations of $\pi_1(N)$.

In the second case, we take $$\Gamma_\text{SW}(N) := \opn{Sym}^g(N), $$ where $g \geq 0 $ denotes the genus of $N$, and $\opn{Sym}^g$ denotes the $g$-fold symmetric product. A Riemann surface structure on $N$ induces a complex structure on $\opn{Sym}^g(N)$. 

Using the ideas of Lagrangian Floer homology and its generalizations, we can combine these manifold-valued invariants for a sequence of surfaces in order to produce homological invariants of 3-manifolds, as proposed above. Furthermore, each of these 3-manifold invariants is conjectured to agree with an invariant coming from gauge theory -- Donaldson theory in the first case, and Seiberg-Witten theory in the second. Finally, it is a longstanding conjecture that the information contained in the Donaldson and Seiberg-Witten invariants agree (at least when the invariants are suitably interpreted), and therefore all four of these frameworks are conjectured to be equivalent, or at least to have strong inter-relationships. 

In addition, many variations on these frameworks lead to other, related manifold invariants. For example, suppose we fix a knot $K\subset S^3$. The 3-manifold invariants above deriving from $\opn{Sym}^g(N)$ are generally called Heegaard Floer invariants, and for the knot $K$, a particularly interesting example is the Heegaard knot Floer homology group $\widehat{HFK}(K)$, which is an invariant of $K$. The first result of our thesis is to construct an analogous knot invariant using representation spaces rather than symmetric products: 

\begin{defn}[symplectic instanton knot homology $\mc{S}(K)$, see Definition \ref{maindef1}] For a knot $K\subset S^3$, we use moduli spaces of $U(2)$ representations to define an abelian group $\mc{S}(K)$, the \textbf{symplectic instanton knot homology}, which is an invariant of $K$.
\end{defn}

This invariant is defined using a formalism developed by Wehrheim and Woodward, and it extends to an invariant of knots in arbitrary 3-manifolds. The main theorem of Chapter 2 uses a result of Ivan Smith to relate $\mc{S}(K)$ and $\widehat{HFK}(K)$ for a special class of knots in $S^3$ called (1,1) knots:

\begin{theorem}[main theorem, chapter 2, see Theorem \ref{main1}]\label{main1intro} For all $(1,1)$ knots $K\subset S^3$, the ranks of $\mc{S}(K)\otimes \C$ and $\HFK(K)\otimes \C$ are equal. 
\end{theorem}

The set of (1,1) knots includes all torus knots, for example. Theorem \ref{main1intro} would appear to be one of the first results indicating that, as folklore would suggest, Heegaard Floer homology groups contain information about $SU(2)$ (or $U(2)$) representations (generally speaking, most work on the Seiberg-Witten = Donaldson philosophy has focused on dimension 4, although many of those techniques could probably be extended to the corresponding 3-manifold invariants).

In fact, we also make the general conjecture:

\begin{conj}The ranks of $\mc{S}(K)\otimes \C$ and $\HFK(K)\otimes \C$ are equal for all knots $K\subset S^3$. 
\end{conj}

Our second result relates invariants from Heegaard Floer homology to another TQFT, namely, Chern-Simons theory, in the form of the Jones polynomial. In particular, we prove the following oddity, a purely topological definition of the Jones polynomial $J(K)$ for 2-bridge knots $K\subset S^3$. Let $\opn{Spin}^c(K)$ denote the set of $\opn{Spin}^c$ structures on the double-branched cover of $K$. For each $\mathfrak{s} \in \opn{Spin}^c(K)$, we define a topological invariant $I(\mathfrak{s})\in \Z$ in terms of Heegaard Floer homology $d$-invariants and Atiyah-Patodi-Singer $\rho$-invariants (alternatively, Casson-Gordon invariants). Let $\sigma(K)$ denote the knot signature. Then, we have (note that there is a bit more notation left undefined, see Chapter 3):

\begin{theorem}[main theorem, chapter 3]\label{main2intro} 
$$i^{-\sigma(K)}q^{3\sigma(K)}J(K)= \sum_{o\in\mathcal{O}}(iq)^{2\sigma(K^{o})} +\left(q^{-1}-q^{1}\right)\sum_{\mfs\in\opn{Spin}^c(K}(iq)^{I(\mfs)}$$
\end{theorem}

Our final result concerns the relationship between the Heegaard Floer homology groups $\widehat{HF}(M)$ of closed 3-manifolds $M$, and their algebraic topology, in particular, their fundamental group. The full nature of this relationship has been an elusive and intriguing question. In joint work with Adam Levine, we prove that if $\widehat{CF}(M)$ is particularly simple, i.e., $M$ is what we call a ``strong $L$-space,'' then we can deduce an interesting property of $\pi_1(M)$, namely, that it admits no left-ordering which is invariant under group multiplication: 

\begin{theorem}[main theorem, chapter 4]\label{main4intro}
If $M$ is a strong $L$-space, then $\pi_1(M)$ is not left-orderable.
\end{theorem}

\chapter{Symplectic instanton invariants and (1,1) knots}\label{chpt1}

In this chapter, we present our results on symplectic instanton knot homology. In the main section, $\S \ref{chpt2mainsection}$, we define the homological knot invariant $\mc{S}(K)$, and prove that it has the same $\C$-rank as $\widehat{HFK}(K)$ for (1,1) knots in $S^3$. The construction of $\mc{S}(K)$ follows directly from the recent and extensive work of Wehrheim and Woodward on quilted Floer homology and in particular their ``Floer field theory,'' and in the first three sections of this chapter, we give a rapid exposition of this work. Most of this exposition is taken directly from the paper \cite{wehrheim2008floerA}, and its updated version \cite{floerfieldtheoryNew}. The interested reader is referred to these papers, and references therein, for the complete story. 

\section{Symplectic topology and the symplectic category}\label{symptop}
We recall the basic definitions:

\begin{defn}[symplectic manifold] A \textbf{symplectic manifold} is a pair $(M,\omega)$ consisting of a smooth, oriented, $2n$-dimensional manifold $M$ and a two-form $\omega\in\Omega^2(M,\R)$
such that 
\begin{enumerate}
\item $\omega$ is closed: d$\omega$ = 0.
\item $\omega$ is pointwise non-degenerate (as a bilinear form on the tangent bundle of $M$).
\end{enumerate}
In most cases, the \textbf{symplectic form} $\omega$ will be left out of the notation, and we will refer to ``the symplectic manifold $M$.''
\end{defn}

By definition, $J\in\operatorname{End}(T^*M)$ is an $\omega$-compatible complex structure on $M$ if $J^2=-\text{Id}$ and $\omega(\cdot,J\cdot)$ is symmetric and positive definite. Let $\mc{J}(M,\omega)$ denote the space of compatible almost complex structures on $(M,\omega)$.  Any $J \in \mc{J}(M,\omega)$ gives rise to a complex structure on the tangent bundle $TM$; the first Chern class $c_1(TM) \in H^2(M,\Z)$ is independent of the choice of $J$.  

\begin{defn}[monotone symplectic manifold]\label{monotone} A symplectic manifold $(M,\omega)$ is \textbf{monotone} if there exists $\tau \geq 0 \in \R$ such that $$ [\omega] = \tau c_1(TM)$$
\end{defn}

\begin{defn}[minimal Chern number] The \textbf{minimal Chern number} $N_M \in \Z$ of $(M,\omega)$ is the non-negative generator of the image of the index map $c_1:\pi_2(M)\to \Z$, given by
$$ u\in \pi_2(M) \mapsto (c_1(TM),u_*[S^2]) \in \Z$$
\end{defn}

\begin{defn}[Lagrangian submanifold] A smooth, oriented, half-dimensional submanifold $L$ of a symplectic manifold $M$ is said to be a \textbf{Lagrangian submanifold}, sometimes simply called a \textbf{Lagrangian}, if $\omega$ vanishes identically when restricted to the tangent bundle of $L$.
\end{defn}

In addition to the index map for $M$, there are two maps associated to a Lagrangian submanifold $L\subset M$, the Maslov index and the action (i.e. symplectic area) maps

$$ I: \pi_2(M,L) \to \Z, \qquad A: \ \pi_2(M,L) \to \R . $$ 
which we will not define (the interested reader can see \cite{oh1993floer} for this and many more details on Lagrangians, and Lagrangian Floer homology). 

We then have 
\begin{defn}[monotone Lagrangian] a Lagrangian submanifold $L\subset M$ is \textbf{monotone} if
$$ 2 A(u) = \tau I(u) \quad \forall u \in \pi_2(M,L) $$
where the $\tau\geq 0$ is (necessarily) that from Definition \ref{monotone}.
\end{defn}

Furthermore, in analogy with the minimal Chern number, one uses the Maslov index map $I$ to define the \textbf{minimal Maslov number} of a Lagrangian $L\subset M$. 

A generalization of a Lagrangian submanifold is a Lagrangian correspondence:

\begin{defn}[Lagrangian correspondence] Let $(M_0,\omega_0)$ and $(M_1, \omega_1)$ denote two symplectic manifolds, and let $\overline{M_0}$ denote $M_0$ with its orientation reversed. A \textbf{Lagrangian correspondence} $L$ from $M_0$ to $M_1$ is a Lagrangian submanifold $L\subset (\overline{M_0}\times M_1 ,-\omega_0\oplus \omega_1)$. 
\end{defn}

Lagrangian correspondences are a simultaneous generalization of Lagrangians and symplectomorphisms: in the first case we take $M_0 = pt$, $M_1 = M$, and in the second case we take $L$ to be the graph of the symplectomorphism. Another natural Lagrangian is the diagonal $\Delta_M\subset \overline{M}\times M$.  Note that strictly speaking, every Lagrangian correspondence is also a Lagrangian submanifold, in the product -- in particular, this directly generalizes the notion of monotonicity to correspondences. 

Lagrangian correspondences play the role of generalized maps between symplectic manifolds. One can define a geometric composition for correspondences, but the result will be another smooth Lagrangian correspondence only in sufficiently nice cases:

\begin{defn}[geometric composition] The \textbf{geometric composition} of Lagrangian correspondences $L_{01} \subset M_0^{-}\times M_1$ and $L_{12}\subset M_1^{-}\times M_2$ is the point set
$$L_{01}\circ L_{12} := \pi_{M_0\times M_2} \left( ( L_{01}\times L_{12}) \cap (M_0\times \Delta_{M_1} \times M_2)\right) \subset M_0 \times M_2 $$
It is called \textbf{transverse} if the intersection is transverse (and hence smooth) and \textbf{embedded} if the projection $\pi_{M_0\times M_2}$ is an embedding of the smooth intersection; if a composition is transverse and embedded, then it results in a smooth Lagrangian correspondence $L_{01}\circ L_{12}\subset M_0^-\times M_2$.
\end{defn}

\begin{defn}[generalized Lagrangian correspondence]
A \textbf{generalized Lagrangian correspondence} $\mathbf{L}$ from $M$ to $N$ consists of a finite sequence of symplectic manifolds $\{M_1,\dots,M_k\}$, with $M_1=M$ and $M_k=N$, and a finite sequence $\ul{L} = \{L_{1,2},\dots,L_{k-1,k}\}$, such that $L_{i,i+1}$ is a Lagrangian correspondence from $M_i$ to $M_{i+1}$.
\end{defn}

The \textbf{algebraic composition} of generalized Lagrangian correspondences $\ul{L}$ and $\ul{L}'$ is given by concatenation $\ul{L}\#\ul{L}' = \left(L_1,\dots,L_m,L_{1}',\dots,L_{m}'\right)$. 

Our main use of Lagrangian correspondences is to define a symplectic category. Because most correspondences are not composable, morphisms will take the form of formal series of Lagrangian correspondences, modulo geometric composition where it is well defined:

\begin{defn}[Symplectic category, $\text{Symp}^{\#}$] 
The objects of \textbf{$\text{Symp}^{\#}$} are smooth symplectic manifolds. The morphisms $\operatorname{Hom}(M_-,M_+)$ of \textbf{$\text{Symp}^{\#}$} are generalized Lagrangian correspondences from $M_-$ to $M_+$ modulo the composition equivalence relation $\sim$ generated by
$$\left(\dots,L_{(j-1)j},L_{j(j+1)},\dots\right)\sim\left(\dots,L_{(j-1)j}\circ L_{j(j+1)},\dots \right) $$ for all 
sequences and $j$ such that $L_{(j-1)j} \circ L_{j(j+1)}$ is transverse and embedded. The composition of morphisms $[\underline{L}] \in \operatorname{Hom}(M,M')$ and $[\underline{L}'] \in \operatorname{Hom}(M',M'')$ is defined by
$$[\underline{L}] \circ [\underline{L}'] := [\underline{L} \# \underline{L}'] \in \operatorname{Hom}(M,M'')$$
The identity $1_M \in \operatorname{Hom}(M,M)$ is the equivalence class $1_M := [\Delta_M]$ of the diagonal $\Delta_M \subset M^- \times M$. 
\end{defn}

Technically, we will use a slightly more elaborate symplectic category for the results in this thesis, which incorporates relative spin structures and monotonicity.

\begin{defn}[relative spin structure]
A \textbf{relative spin structure} on a bundle $E \to M$ with respect to a map $M \to N$ is a relative trivialization of the second Stiefel-Whitney class $w_2(E) \in H^2(M,\Z_2)$. (In particular, $E$ is relatively spinable if and only if $w_2(E)$ lies in the image of $H^2(N,\Z_2)\to H^2(M,\Z_2)$). 
\end{defn}

Then, the symplectic category which we will need is:

\begin{defn}[monotone symplectic category, $\operatorname{Symp}^{\#}_{\tau}$]
The \textbf{monotone symplectic category} $\operatorname{Symp}^{\#}_{\tau}$ denotes the category with monotone symplectic manifolds with monotonicity constant $\tau$ as objects, and equivalence classes of generalized Lagrangian correspondences with relative spin structures, with minimal Maslov number at least three, as morphisms. (Note that the empty set is allowed as an an object of $\operatorname{Symp}^{\#}_{\tau}$). 
\end{defn}

\section{Lagrangian Floer homology, quilted Floer homology, and the categorification functor}

To define Lagrangian Floer homology, suppose we have two Lagrangians $L_1,L_2\subset M$. Lagrangian Floer homology is a $\Z/2\Z$-graded abelian group associated to this pair, written $HF(L_1,L_2)$. The definition of Floer homology is both elegant and revolutionary, but is also involved, and we will not need a detailed exposition for any of the results in this thesis. An interested reader should turn to \cite{wehrheim2010functoriality} and \cite{oh1993floer} for details. In brief, if $L_1$ and $L_2$ are compact and have transverse intersection, then $L_1\cap L_2$ consists of a finite set of points. In this situation, we can give a preliminary description of Lagrangian Floer homology, as follows:

\begin{defn}[Lagrangian Floer homology, preliminary]
If $L_1,L_2\subset M$ are compact Lagrangians in $M$ with transverse intersection, then the \textbf{Lagrangian Floer homology} $HF(L_1,L_2)$ is the homology of a chain complex $CF(L_1,L_2)$. As an abelian group, $CF(L_1,L_2)$ is generated by a distinguished basis $[x_i]$, where $\{x_i\} = L_1\cap L_2$ denotes the set of intersection points between $L_1$ and $L_2$. The $\Z/2\Z$ grading arises from the function $\{x_i\}\to\{\pm 1\}$ which maps each intersection point to its sign. The differential $\partial: CF(L_1,L_2)\to CF(L_1,L_2)$ is the $\Z$-linear map defined in terms of this basis by associating certain integers $n(x_i,x_j)$ to pairs of intersection points, and defining $\partial([x_i]) = n(x_i,x_j) [x_j]$.
\end{defn}

The integer $n(x_i,x_j)$ is a signed count of certain disks $\mathbb{D} \to M$ with boundary on $L_1$ and $L_2$, which ``cancel'' $x_i$ and $x_j$, in the sense of a Whitney move. The disks which contribute to $n(x_i,x_j)$ are essentially those to which $\omega$ restricts as an area form. Furthermore, the analysis necessary to proving that $HF(L_1,L_2)$ is well-defined famously requires the choice of an $\omega$-compatible almost-complex structure $J$ on $M$; from this point of view, the relevant disks are those which are holomorphic with respect to $J$. In other words, there is a map $u:\mathbb{D}\to M$, parameterizing the disk in $M$, whose differential $du$ intertwines the (differential of) complex multiplication by $i$ on $\mathbb{D}$ with multiplication by $J$ on $T^*M$.

Note that if we move a Lagrangian submanifold by a special subclass of isotopies which preserve $\omega$, called \textbf{Hamiltonian isotopies}, then the traces of arcs in the Lagrangian under the isotopy will be $J$-holomorphic, for appropriate $J$. In this sense, $HF(L_1,L_2)$ serves as an algebraic device for capturing ``symplectically un-cancellable'' intersections of $L_1$ and $L_2$. Indeed, the group $HF(L_1,L_2)$ gives a strict generalization of the algebraic intersection number $I(L_1,L_2) = ([L_1]\cup [L_2])([M])$ (the ``algebraically un-cancellable intersections), since the algebraic intersection number is given by the Euler characteristic of Floer homology, $$I(L_1,L_2) = \chi(HF(L_1,L_2))$$

In the rest of this section, we present a very rapid overview of the relevant definitions and results concerning quilted Floer homology. For the details, see \cite[$\S 4$]{floerfieldtheoryNew}, and references therein.

Let $M$ be a $\tau$-monotone symplectic manifold, as defined in $\S$\ref{symptop}.

\begin{defn}[generalized Lagrangian manifold]
A \textbf{generalized Lagrangian submanifold} of $M$ is a generalized Lagrangian correspondence from a point to $M$, that is, a sequence $L_{-s(-s+1)},\dots,L_{(-1)0}$ of correspondences from $M_{-s} = \opn{pt}$ to $M_0 = M$. We say that a generalized Lagrangian correspondence satisfies a certain property (simply-connected, compact, etc.) if each correspondence in the sequence satisfies that property.
\end{defn}

Using their holomorphic quilt technology, Wehrheim and Woodward generalize the definition of Lagrangian Floer homology to define the \textbf{quilted Floer homology} of two generalized Lagrangians $\ul{L}_0$ and $\ul{L}_1$, which we will continue to write as $HF(\ul{L}_0,\ul{L}_1)$. These groups serve as the Hom sets for an extended Donaldson-Fukaya category, which we now define:

\begin{defn}[extended Donaldson-Fukaya category] $\opn{Don}^{\#}(M)$, the \textbf{extended Donaldson-Fukaya category}, is the category whose
\begin{enumerate}
\item objects are compact, oriented, simply-connected generalized Lagrangian submanifolds of $M$
\item morphisms from an object $\ul{L}_0$ to an object $\ul{L}_1$ are quilted Floer homology classes:
$$\opn{Hom}(\ul{L}_0,\ul{L}_1) = HF(\ul{L}_0,\ul{L}_1)$$
\item composition and identities are defined by counting holomorphic quilts with strip-like ends and Lagrangian boundary and seam conditions as in \cite{wehrheim2010functoriality} (this is the quilted generalization of relative invariants defined by counting holomorphic strips. See \cite{wehrheim2010functoriality} for an overview). 
\end{enumerate}
\end{defn}

\begin{defn}[Functors for Lagrangian correspondences] Let $M_0,M_1$ be $\tau$-monotone symplectic manifolds. For any compact, oriented, simply-connected spin correspondence $L_{01} \subset M_0^- \times M_1$ the functor $$ \Phi(L_{01}) : \opn{Don}^{\#}(M_0) \to \opn{Don}^{\#}(M_1)$$ is defined on objects by $$\left(L{-s(-s+1)},\dots,L_{(-1)0)} \right)\mapsto \left(L_{-s(-s+1)},\dots,L_{(-1)0}, L_{01}\right)$$ (i.e., algebraic composition). 
On morphisms $\Phi(L_{01})$ is defined by counting holomorphic quilts of the form in \cite[p.37, Figure 4]{floerfieldtheoryNew}, i.e. by counting (quilted) pairs of pants.
\end{defn}

The main result of this section packages together the previous quilted Floer homology constructions to construct a ``categorification functor'' from the symplectic category to the category $\opn{Cat}$ of (small) categories. This ``black-boxes'' all the details, analytic and otherwise, in the Floer homology constructions, and therefore to define a field theory using Floer homology, one only has to focus on the question of which symplectic manifolds and Lagrangian correspondences one would like to use.

\begin{theorem}[categorification functor]\label{catfunc} For any $\tau > 0$, the maps $$M \mapsto \opn{Don}^{\#}(M), ~~ [\ul{L}_{01}]\mapsto [\Phi(\ul{L}_{01})]$$ define a \textbf{categorification functor} $\opn{Don}^\# : \opn{Symp}_\tau\to \opn{Cat}$.
\end{theorem}

\section{Floer field theory, following Wehrheim and Woodward}

\subsection{Decorated cobordism categories}
The field theories defined by Wehrheim and Woodward are invariants of 2 and 3 manifolds, equipped with additional bundle structure. This is formalized in terms of \textbf{decorated cobordism categories}, which we define in this subsection. 
	
Fix an integer $r > 0$, and let $P$ be a principal $U(r)$-bundle over a compact, connected surface $X$. We make the following definitions (the first is just the standard definition of degree):

\begin{defn}[degree]\label{degree}
The \tbf{degree} of $P$ is the integer deg$(P) = (c_1(P),[X]) \in \Z $. 
\end{defn}

\begin{defn}[decorated surface]\label{decoratedsurface}
A \textbf{decorated surface of rank \textit{r} and degree \textit{d}} consists of
\begin{enumerate}
\item a compact, smooth, oriented 2-manifold $X$
\item a principal $U(r)$-bundle $P \to X$ with deg$(P) = d$
\item a connection $\delta$ on det$(P)$
\end{enumerate}
\end{defn}

\begin{defn}[decorated cobordism]\label{decoratedcobordism}
A \tbf{decorated cobordism between decorated surfaces $(X_{\pm},P_{\pm},\delta_{\pm})$ of rank \textit{r} and degree \textit{d}} consists of
\begin{enumerate}
\item a compact connected oriented Riemannian three-manifold $Y$ with partitioned boundary $\partial Y = X_- \cup X_+$
\item a principal $U(r)$-bundle $P \to Y$
\item a constant curvature connection $\delta$ on det$(P)$
\item isomorphisms of the restriction of $(P,\delta)$ to $(\partial Y )_{\pm}$ with $(P_{\pm}, \delta_{\pm})$
\end{enumerate}
\end{defn}

\begin{defn}[(2+1)-dimensional decorated cobordism category, $\dcob{2}^{(r,d)}$]\label{decoratedcobordismcategory} The \textbf{(2+1)-dimensional decorated cobordism category} \textbf{$\dcob{2}^{(r,d)}$} is the category whose objects are connected, rank $r$, degree $d$ decorated surfaces without boundary, and whose morphisms are rank $r$, degree $d$ decorated cobordisms, modulo diffeomorphisms which are the identity on the boundary, and pull back the relevant bundle structure. 
\end{defn}

\subsection{Simple cobordisms, Heegaard splittings, and invariance}\label{handlesection}

In general, a (weak) (\textit{d}+1)-dimensional $\mc{C}$-valued topological field theory (TFT) will be a functor from the (\textit{d}+1)-dimensional cobordism category, possibly decorated by extra structure, into another category $\mc{C}$. In our case, we will only need the (2+1)-dimensional decorated cobordism category from Definition \ref{decoratedcobordismcategory}. (Furthermore, a full (rather than weak) TFT would also include invariants for disconnected $d$-manifolds, and for diffeomorphisms of $d$-manifolds). Therefore, for this thesis, we make the following definition: 

\begin{defn}[weak (\textit{2+1})-dimensional $\mc{C}$-valued topological field theory]\label{weak} For integers $r,d > 0$ and a category $\mc{C}$, a \textbf{weak (\textit{2+1})-dimensional $\mc{C}$-valued topological field theory (TFT) of rank \textit{r} and degree \textit{d}} is a functor from $\dcob{2}^{(r,d)}$ into $\mc{C}$.  
\end{defn}

Thus, a weak TFT will associate functor-valued invariants to any oriented, compact, 3-dimensional cobordism with two (non-empty) boundary components, once appropriate bundle data is chosen. Furthermore, these invariants will be compatible with cutting and gluing of cobordisms, and so one could hope that to define an entire TFT, it might suffice to give its value on a basic ``generating set'' of cobordisms, which would then uniquely determine the remaining theory by composition. There would be a strong constraint on the invariants assigned to the basic pieces, namely, whenever distinct gluings of these pieces yielded the same 3-manifold, the corresponding compositions of functors would have to agree. 

In fact, a version of this strategy is already evident in Ozsv\'ath and Szab\'o's definition of the Heegaard Floer homology $HF(M)$ of a closed, oriented 3-manifold $M$; to define $HF(M)$, one first chooses a particular decomposition of $M$, called a Heegaard splitting. In a Heegaard splitting, $M$ is decomposed into exactly two pieces; furthermore, each piece is required to be a handlebody, which is a particularly simple cobordism with one boundary component. A single 3-manifold $M$ admits many distinct Heegaard splittings, and the details of the construction of $HF(M)$ depend crucially on the choice of splitting. Most of Ozsv\'ath and Szab\'o's original paper \cite{ozsvath2004holomorphicDiskInvariants} defining $HF(M)$ is devoted to proving that their invariant is actually independent of all the choices necessary for its construction, beyond the 3-manifold itself.

To define their invariants, Wehrheim and Woodward introduce a framework which generalizes this Heegaard Floer homology picture, which we will now describe. (Note that we will intentionally omit many of the details, definitions, and proofs, for the sake of brevity, and the interested reader should refer to \cite[\S 2]{wehrheim2008floerA} for a complete discussion). In summary, Wehrheim and Woodward allow arbitrary decompositions of their 3-dimensional cobordisms into pieces which they call simple cobordisms. They then use Cerf theory to derive a general set of conditions which are necessary and sufficient for a ``partial TFT,'' defined only on simple cobordisms, to yield a consistent TFT on all 3-dimensional cobordisms via gluing. Again, as with most of the expository material in this chapter, the majority of the following exposition is taken rather directly from \cite{wehrheim2008floerA}. 

We begin with the general definitions of the relevant cobordisms. Let $X_-,X_+$ be compact, connected, oriented manifolds of dimension $d \geq 1$, and let $Y$ be a compact, oriented cobordism from $X_-$ to $X_+$, i.e., $Y$ is a manifold with boundary of dimension $d + 1$ and $X_+$, respectively $X_-$, is the component of the boundary $\partial Y = X_- \cup X_+$  on which the given orientation agrees, respectively disagrees, with the orientation induced by the orientation on $Y$. 

Wehrheim and Woodward's framework is based around cobordisms equipped with a Morse function $f$, with some extra data and conditions; together these form a Morse datum:

\begin{defn}[Morse datum]\label{morsedatum}
A \textbf{Morse datum} for $Y$ consists of a pair $(f,\underline{b})$ of a Morse function $f: Y \to \R$ and an ordered tuple $\underline{b} = (b_0 < b_1 < \cdots < b_m) \subset R_{m+1}$ such that
\begin{enumerate}
\item $X_- = f^{-1}(b_0)$ and $X_+ = f^{-1}(b_m)$ are the sets of minima, resp. maxima, of $f$,
\item each level set $f^{-1}(b)$ for $b \in \R$ is connected, that is, $f$ has no critical points
of index $0$ or $d+1$,
\item $f$ has distinct values at the (isolated) critical points, i.e. it induces a bijection
$\operatorname{Crit}f \to f(\operatorname{Crit}f)$ between critical points and critical values,
\item $b_1, \dots , b_{m-1} \in \R \backslash f(\operatorname{Crit}f)$ are regular values of $f$ such that each interval
$(b_{i-1},b_i)$ contains at most one critical value of $f$.
\end{enumerate}
\end{defn} 

Note that, given a Morse function $f$ satisfying conditions $1$-$3$ in Definition \ref{morsedatum}, there always exists a choice of $b_1 < \dots < b_{m-1}$ satisfying condition 4.

The distinguished cobordisms in the Wehrheim-Woodward theory are defined in terms of Morse data:

\begin{defn}[simple cobordism]\label{simplecobordism}
We call $Y$ a \textbf{simple cobordism} if it admits a Morse datum $(f,\underline{b})$ where $f$ is a Morse function with at most one critical point (and hence we can choose $\underline{b} = (\operatorname{min}f,\operatorname{max}f)$). 
\end{defn}
\begin{defn}[cylindrical cobordism]\label{cylindricalcobordism}
We call $Y$ a \textbf{cylindrical cobordism} if it admits a Morse datum $(f,\underline{b})$ where $f$ is a Morse function with no critical point (and $\underline{b} = (\operatorname{min}f, \operatorname{max}f))$. 
\end{defn}

Note that if the simple cobordism $Y$ contains no critical point then it is always a cylindrical cobordism; in that case the boundary components $X_-$ and $X_+$ are diffeomorphic to the same manifold $X$, and $Y$ is diffeomorphic to the cylinder $X \times [0,1]$. Otherwise, $Y$ contains a single critical point, with index $k \in \{1,\dots,d\}$, and $X_-$ is obtained from $X_+$ by attaching a handle $S_{k-1}\times B_{d-k}$, via an \textbf{attaching cycle} $S_{k-1}\times S_{d-k} \to X_-$, given by the intersection of the unstable manifold (for some choice of a metric on $Y$) for the unique critical point with $X_{-}$. Conversely, $X_{-}$ can be obtained from $X_+$ by attaching a handle of opposite index to an attaching cycle in $X_+$.

Two additional classes of cobordisms which can be defined in terms of Morse functions are:

\begin{defn}[compression body]\label{compressionbody}
We say that a three-dimensional cobordism $Y$ is a \textbf{compression body} if $Y$ can be obtained from $\partial Y_{-}$ or $\partial Y_{+}$ by adding only 1-handles or adding only 2-handles, that is, $Y$ admits a Morse function with minimum $\partial Y_{-}$, maximum $\partial Y_{+}$, and critical points of all of index 1 or all of index 2. 
\end{defn}

\begin{defn}[handlebody]\label{compressionbody}
We say that a three-dimensional cobordism $Y$ is a \textbf{handlebody} if $Y$ is a compression body such that one of $(\partial Y)_{\pm}$ is empty. The \textbf{genus} of $Y$ is defined to be the number of critical points of a Morse function on $Y$; it follows that $\partial Y$ is a genus $g$ surface $\Sigma_g$. 
\end{defn}

For completeness, we can now define
\begin{defn}[Heegaard splitting, Heegaard surface]\label{heegaard} A \textbf{Heegaard splitting} of an oriented 3-manifold $M$ is a decomposition $M = Y_1 \cup_{\Sigma_g} Y_2$, where $Y_1$ and $Y_2$ are each genus $g$ handlebodies ($g$ is also called the \textbf{genus} of the Heegaard splitting). Furthermore $\Sigma_g = \partial Y_1 = \partial Y_2$ is called the \textbf{Heegaard surface}.
\end{defn}


Since all smooth, compact, oriented manifolds with boundary admit Morse functions, any smooth, compact, oriented cobordism with two non-empty, connected boundary components can be decomposed into a finite sequence of simple cobordisms. To move between different decompositions of the same cobordism, we use the following relationships between sequences of simple cobordisms (i.e., when glued up, each sequence in the pairs listed below have the same diffeomorphism type). In the notation below, $\partial Y_i = X_{i-1}\cup X_{i}$. 

\begin{defn}[critical point cancellation]\label{criticalpointcancellation}
In which two simple cobordisms $Y_i, Y_{i+1}$, which carry critical points of adjacent indices whose attaching cycles (for some choice of a metric) in $X_i$ intersect transversally in a single point, are replaced by the cylindrical cobordism $Y_i\cup_{X_i} Y_{i+1} \cong X_{i-1}\times [b_{i-1}, b_{i+1}]\cong X_{i+1}\times [b_{i-1},b_{i+1}]$
\end{defn}

\begin{defn}[critical point reversal]
In which two simple cobordisms $Y_i, Y_{i+1}$, which carry critical points of index $k$ and $l$ whose attaching cycles (for some choice of a
metric) in $X_i$ do not intersect, are replaced by two simple cobordisms $Y_{i}',Y_{i+1}'$, which carry critical points of index $l$ and $k$ whose attaching cycles in $X_{i}'$ do not intersect, such that $Y_{i}\cup_{X_i} Y_{i+1} =Y_{i}'\cup_{X_{i}'} Y_{i+1}'$ up to a diffeomorphism that fixes the boundary $X_{i-1}\cup X_{i+1} = X_{i-1}'\cup X_{i+1}'$. 
\end{defn}

\begin{defn}[cylinder gluing]
In which two simple cobordisms $Y_i,Y_{i+1}$, one of which is cylindrical, are replaced by the simple cobordism $Y_{i}\cup_{X_i} Y_{i+1}$. 
\end{defn}

Using these moves, Wehrheim and Woodward prove the following invariance theorem, which gives conditions for a TFT defined just on simple cobordisms to extend to all cobordisms: 

\begin{theorem}\label{invariance}
Any partial functor $\dcob{2}^{(r,d)} \to \mc{C}$, which associates
\begin{enumerate}
\item to each compact, connected, oriented $d$-manifold $X$, an object $C(X) \in \text{Obj}(\mc{C})$,
\item to each equivalence class of compact, connected, oriented simple cobordism $Y$ from $X_-$ to $X_+,$ a morphism $\Phi(Y)$ from $C(X_-)$ to $C(X_+)$,
\item to the trivial cobordism $[0,1]\times X$ the identity morphism $1_C(X)$ of C(X),
\end{enumerate}

and satisfies the Cerf relations

\begin{enumerate}
\item If $Y_1$ from $X_0$ to $X_1$ and $Y_2$ from $X_1$ to $X_2$ are simple cobordisms such that $Y_1 \cup_{X_1} Y_2$ is a cylindrical cobordism via critical point cancellation, then
$$\Phi(Y_1) \circ \Phi(Y_2) = \Phi(Y_1 \cup_{X_1} Y_2) $$
\item If $Y_1,Y_2$ and $Y'_1,Y'_2$ are simple cobordisms related by critical point reversal, then
$$ \Phi(Y_1) \circ \Phi(Y_2) = \Phi(Y'_1) \circ \Phi(Y'_2) $$
\item If $Y_1,Y_2$ are simple cobordisms, one of which is cylindrical, then
$$ \Phi(Y_1) \circ \Phi(Y_2) = \Phi(Y_1 \cup_{X_1} Y_2)$$
\end{enumerate}
extends to a unique weak (\textit{2+1})-dimensional $\mc{C}$-valued topological field theory.
\end{theorem}

\begin{proof} See \cite{wehrheim2008floerA}, pages 5 and 6, and the comment following the theorem statement on page 6.
\end{proof}
\subsection{Moduli spaces of $U(r)$ connections and symplectic-valued field theories}

By Theorem \ref{invariance}, to define a weak $\mc{C}$-valued (2+1)-dimensional TFT, it suffices to assign functors in $\mc{C}$ to all simple cobordisms, and prove that these functors satisfy the Cerf relations. Ultimately, we are after functor-valued TFT's, i.e. we would like to take $\mc{C}$ to be $\opn{Cat}$, the category of categories. The construction will be factored into two stages: first, we define a partial functor $$\mc{M}: \dcob{2}^{(r,d)} \to \opn{Symp}_\tau^\#,$$ which satisfies the Cerf relations and therefore can be extended to an honest field theory with values in $\opn{Symp}_\tau^\#$; then, we apply the Floer homology categorification functor from Theorem \ref{catfunc}. 

Recall that the objects and morphisms of $\dcob{2}^{(r,d)}$ are manifolds equipped with a principle bundle $P$ and a connection $\delta$ on $\opn{Det}(P)$. To such a decorated manifold the functor $\mc{M}$ associates the moduli space of central curvature connections on $P$ with determinant equal to $\delta$. Generally speaking, these moduli spaces are finite dimensional varieties (sometimes singular), defined as the quotient of an infinite dimensional affine space of connections by the action of an infinite dimensional Lie group of bundle automorphisms, called the gauge group. 

We will not give any details for this general case; for these, the interested reader should refer to \cite[$\S 3.2$]{floerfieldtheoryNew}. Instead, in the next section, we will give a more topological description of these connection spaces in the case that $(r,d) = (2,1)$. However, for completeness, and to make contact with Wehrheim and Woodward's notation, we first give the general statement proved in \cite{floerfieldtheoryNew}:

\begin{defn}[moduli spaces of central curvature connections for a decorated surface]\label{modsurface}
For $(X,P,\delta)$ a decorated surface, define 

$$M(X) := M_\delta(X,P)$$ 

to be the moduli space of central curvature connections on $P$ with determinant $\delta$. 
\end{defn}

\begin{defn}[moduli spaces of central curvature connections for a decorated cobordism]\label{modcobordism}
For $(Y,P,\delta)$ a decorated cobordism with boundary $(X_{\pm},P_\pm,\delta_\pm)$ define 

$$L(Y) := L_\delta(Y,P)\subset M(X_-)\times M(X_+)$$

to be the image, under restriction to the boundary, of the moduli space of central curvature connections on $P$ with determinant $\delta$.
\end{defn}

In this notation, the following is the main theorem proved in \cite{floerfieldtheoryNew}, demonstrating that moduli spaces of connections yield a symplectic-valued TFT: 

\begin{theorem}\label{repTFT}
Suppose that $r$ is coprime to $d$.
\begin{enumerate}
\item For any decorated surface $X$ with rank $r$ and degree $d$, $M(X)$ is a smooth compact 1-connected manifold and admits a canonical monotone symplectic form with monotonicity constant $\tau^{-1} = 2r$. 
\item For any decorated \textit{simple} cobordism $Y$ with rank $r$ and degree $d$, $L(Y)$ is a smooth Lagrangian correspondence and admits a unique relative spin structure.
\item The assignments $$X\mapsto M(X), ~ Y\mapsto L(Y)$$ satisfy the Cerf relations of Theorem \ref{invariance}, and therefore define a topological field theory $$\mc{M}: \dcob{2}^{(r,d)} \to \operatorname{Symp}^{\#}_{1/2r}$$
\end{enumerate}
\end{theorem}

\begin{rem}
We emphasize that for part 2. of Theorem \ref{repTFT}, it is crucial that $Y$ be a \textit{simple} cobordism; for a general decorated cobordism $Y$, the moduli space of connections on $Y$ will generally not give a smooth submanifold when restricted to the boundary moduli spaces. 
\end{rem}

\subsection{Moduli spaces of twisted $SU(2)$ representation}\label{modsu2section}

Rather than give the precise definitions of the moduli spaces of central curvature $U(r)$ connections from the previous section, in this section we will give an alternative topological description. For simplicity we restrict to the case that $(r,d) = (2,1)$, though there is an analogous topological description for every rank and degree. 

An original, in-depth reference for this material (and much more) is \cite{atiyah1983yang}, but Wehrheim and Woodward also include this alternative description in their paper, so readers looking for proofs of the statements in this section can turn to \cite[p. 29]{floerfieldtheoryNew}, and references therein. 

To begin, let $\Sigma_g$ be a closed, oriented surface of genus $g$. Choose a basepoint $p\in \Sigma_g$, and let $\gamma$ be a small loop in $\Sigma \backslash \{p\}$ which is freely isotopic to the puncture. Then

\begin{defn}[moduli space $\mc{R}_g$ of twisted $SU(2)$ representations for $\Sigma_g$]\label{twisted} The \textbf{moduli space $\rg$ of twisted $SU(2)$ representations for a surface $\Sigma_g$ of genus $g$} is  $$ \rg : = \{ \rho : \pi_1 (\Sigma \backslash \{p\})\to SU(2) : \rho(\gamma) = -\mathbb{I} \}/SU(2)$$
Here, $\mathbb{I}$ denotes the identity matrix in $SU(2)$, and the quotient is by conjugation. Although $\gamma$ only defines a conjugacy class in $\pi_1 (\Sigma \backslash \{p\})$, $-\mathbb{I}$ is central, so the condition that $\rho(\gamma)=-\mathbb{I}$ is well-defined. For the same reason, this condition is conjugation invariant, so the conjugation action is well-defined. 
\end{defn}

As noted above, $\mc{R}_g$ gives another description of the moduli spaces of central curvature connections with fixed determinant:

\begin{theorem}[see \cite{atiyah1983yang}, $\S 6$]\label{repconn} For a decorated surface $(X,P,\delta)$ with genus $g$, rank 2, and degree 1, the association $\alpha \mapsto \opn{Mon}_\alpha$, sending a connection to its monodromy mapping, leads to a diffeomorphism $$M(X) \cong \mc{R}_g,$$ where $M(X)$ is the moduli space from Definition \ref{modsurface}. (Further, this diffeomorphism is natural with respect to diffeomorphisms of decorated surfaces). 
\end{theorem}

Therefore, the following properties of $\mc{R}_g$ are a direct corollary of the general discussion in \cite{floerfieldtheoryNew}:

\begin{theorem}[see \cite{floerfieldtheoryNew}\label{rprop}, $\S 3.2 -\S 3.3$]\label{modsu2props} Fix $g\geq 1$. 

\begin{enumerate}
\item $\rg$ is a smooth, oriented, compact manifold of dimension $6g-6$.
\item $\rg$ has a canonical symplectic form $\omega$.
\item $(\rg,\omega)$ is monotone with minimal Chern number 2. 

\end{enumerate}
\end{theorem}

By part 1 of Theorem \ref{modsu2props}, when $g = 1$ (so $\Sigma_g$ is a torus), the moduli space is 0 dimensional. In fact, it consists of a single point:

\begin{theorem}[twisted $SU(2)$ representations on the torus]\label{modsu2torus} $ \mc{R}_1 \cong \text{pt}.$ 
\end{theorem}

\begin{proof} The content of the theorem is that there is, up to conjugation, a unique homomorphism $\pi_1(T^2\backslash \{p\})\to SU(2)$ assigning $-\mathbb{I}\in SU(2)$ to a small loop $\gamma$ around $p$. To prove this, choose a standard basis $x,y$ for the rank 2 free group $\pi_1(T^2\backslash\{p\})$, and note that $\gamma = [x,y]$ (the commutator of $x$ and $y$). At this point, one can show directly that any pair of matrices $(A,B)$ in $SU(2)$ satisfying $[A,B] = -\mathbb{I}$ can be mutually conjugated to the pair $(I,J)$, where $I$ and $J$ are the standard matrices representing the corresponding unit quaternions. 
\end{proof}

Note that the moduli space $\mc{R}_0$ is empty, i.e. the sphere admits no twisted $SU(2)$ representations.

In addition to the moduli spaces for surfaces, the connection moduli spaces $L(Y)$ associated to decorated cobordisms (see Definition \ref{modcobordism}) also admit a topological description, in terms of Lagrangian correspondences between moduli spaces of twisted $SU(2)$ representations. For brevity, we will not include a general discussion of these spaces, but we will describe a special case, in which these correspondences actually reduce to ``classical'' Lagrangian submanifolds. Namely, let $Y$ be a compression body, and furthermore suppose that $Y$ goes from a torus to a higher-genus surface, i.e. $\partial Y = T^2 \cup \Sigma_g$ with $g\geq 1$. Choose basepoints $p_0\in T^2$ and $p_1\in \Sigma_g$, and let $\ell\subset Y$ be a connected arc whose intersection with $T^2$ and $\Sigma_g$ is given by the sets $\{p_0\}$ and $\{p_1\}$, respectively. Let $\gamma\subset (Y\backslash \ell)$ be a meridian of the arc $\ell$ (i.e., $\gamma$ gives a section of the normal bundle to $\ell$, intersecting the normal fibers with multiplicity 1). 

\begin{defn}\label{replag} Define $\L(Y,\ell) \subset \mc{R}_g$ to be the subspace of conjugacy classes of representations of $\pi_1(\Sigma_g\backslash \{p_1\})$ in $\mc{R}_g$ which extend to representations of $\pi_1(Y\backslash \ell)$, and which send any loop in the conjugacy class of $\gamma$ to $-\mathbb{Id}$. 
\end{defn}

\begin{lemma}\label{ell}
$\L(Y) := \L(Y,\ell)$ is independent of the choice of $\ell$.
\end{lemma}

\begin{proof}
Because $Y$ is a compression body, it is obtained from $\Sigma_g$ by attaching 2-handles to $g-1$ disjoint simple closed curves $\{\alpha_i\} \subset \Sigma_g$ (see the discussion following Definition \ref{cylindricalcobordism} for the definition of an attaching handle). In particular, the inclusions $\Sigma_g \into \partial Y \into Y$ induce a surjection $\pi_1(\Sigma_g)\onto \pi_1(Y)$. Therefore, all loops in $Y$ can be isotoped into a collar neighborhood of the $\Sigma_g$ component of the boundary, and we can take the collar small enough so that within the collar neighborhood, $\ell$ is given by the product of the collar by $p_1$. 
\end{proof}

In fact, this proof is easily extended to a stronger statement that clearly implies Lemma \ref{ell}, by giving an explicit $\ell$-independent description of the subspace $\L(Y,\ell) \subset \mc{R}_g$:

\begin{lemma}\label{attach}
Let $\{\alpha_i\} \subset \Sigma_g$ be a set of attaching curves for $Y$, as in the proof of Lemma \ref{ell}. Each $\alpha_i$ defines a conjugacy class in $\pi_1(\Sigma_g)$, and let $\mc{A}\subset \pi_1(\Sigma_g)$ denote the union of these conjugacy classes over all $\alpha_i$. Then  $$\L(Y) = \{ [\rho] \in \mc{R}_g \text{ such that } \rho(\mc{A}) = \mathbb{I}\in SU(2) \} $$ 
(again, this condition is conjugation invariant, and therefore well-defined, similarly to the discussion in Definition \ref{twisted}). 

\end{lemma}

As with $M(\Sigma_g)$ and $\mc{R}_g$, $L(Y)$ and $\L(Y)$ give different definitions of the same space:

\begin{theorem}[see \cite{floerfieldtheoryNew}]\label{equal}
Assuming still that $Y$ is a compression body with $\partial Y = T^2\cup \Sigma_g$, let $(Y,P,\delta)$ be a rank 2, degree 1 decorated cobordism structure on $Y$. Let $$L(Y) \subset M(T^2)\times M(\Sigma_g)$$ be the Lagrangian correspondence from Definition \ref{modcobordism}. Then

\begin{enumerate}

\item $L(Y)$ is diffeomorphic to its projected image $\pi_{M(\Sigma_g)}(L(Y)) \subset M(\Sigma_g)$ (this is simply because, as seen in Theorem \ref{modsu2torus}, $M(T^2) = \opn{pt}$), and therefore we can (and will, for the rest of this theorem) view $L(Y)$ as a subspace of $M(\Sigma_g)$. 

\item The diffeomorphism $M(\Sigma_g) \cong \mc{R}_g$ from Theorem \ref{repconn} yields a diffeomorphism between $L(Y)\subset M(\Sigma_g)$ and the subspace $\L(Y)\subset \mc{R}_g$ from Definition \ref{replag}:
$$ L(Y) \subset M(\Sigma_g) \cong \L(Y) \subset \mc{R}_g$$
\end{enumerate}
\end{theorem}

Thus, as with Theorem \ref{rprop}, the following properties of $\L(Y)\subset \mc{R}_g$ follow from the general discussion in \cite{floerfieldtheoryNew}:
\begin{theorem}[see \cite{floerfieldtheoryNew}, $\S 3.3$]\label{lprop} Fix $g\geq 2$, and let $Y_g$ denote a compression body with boundary $\partial Y = T^2 \cup \Sigma_g$. Then 
\begin{enumerate}
\item $\L(Y)$ is a simply-connected, Lagrangian submanifold of $\rg$, of half-dimension $3g-3$; in particular, it is homeomorphic to the product of $g-1$ 3-spheres. 
\item $\L(Y)$ is a monotone Lagrangian manifold, with minimal Maslov index 4.
\end{enumerate}
\end{theorem}

Note the importance of the role played by Theorem \ref{modsu2torus} (the fact that the torus moduli space is a point) in Theorems \ref{equal} and \ref{lprop}. In the general case that $Y$ is a compression body from $\Sigma_{g_1}$ to $\Sigma_{g_2}$ with $1 < g_1 < g_2$, one can still project $L(Y)$ into $M(\Sigma_{g_2})$, but the image will not be Lagrangian, as one can check simply on dimension grounds.

\subsection{Category-valued and group-valued field theories}

By combining Theorem \ref{repTFT} and the categorification functor from Theorem \ref{catfunc}, we immediately deduce the following theorem/construction of category-valued field theories from twisted representation spaces (alternatively, moduli spaces of fixed-central-curvature connections):
\begin{theorem}[topological field theories from representation spaces, $\mc{FR}^{(r,d)}$]\label{rdFldTheory}
For any coprime integers $r$ and $d$, $r > 0$, the maps 
\begin{align*}
X &\mapsto C(X) := Don^{\#}(M(X)) \\
Y &\mapsto \Phi(Y ) := \Phi(\underline{L}(Y ))
\end{align*}
define a weak topological field theory ${\mathcal{FR}^{r,d}}$ from $\dcob{2}^{(r,d)}$ to the category $\operatorname{Cat}$ of (categories, isomorphism classes of functors), 

$$ \mathcal{FR}^{r,d}: \dcob{2}^{(r,d)}\to \operatorname{Cat} $$

In this thesis, we will only need the simplest (non-abelian) case of this invariant, with $r = 2$ and $d=1$, which we write as $$\mathbf{\mathcal{FR}} := \mathcal{FR}^{2,1}: \dcob{2}^{(2,d)}\to \operatorname{Cat}$$
\end{theorem}
%

Note that here and in the remaining content of this chapter, we will fix $r=2,d=1$. Although in its full generality, $\mathcal{FR}$ gives functor-valued invariants of 3-dimensional cobordisms, for the purposes of this thesis we are interested in a particularly simple special case, in which the information in these functors is actually captured by a single abelian group. Recall that by Theorem \ref{modsu2torus}, when $X$ has genus 1 (i.e., $X=T^2$), the moduli space $M(X)$ consists of a single point, and therefore has a unique Lagrangian submanifold ($M(X)$ itself). This suggests that in the context of $\mc{FR}$, the role of closed 3-manifolds is actually played by oriented cobordisms with two torus boundary components, as we have already seen implicitly in Theorem \ref{lprop}. This inspires the following definition:

\begin{defn}[Lagrangian $U(2)$ Floer homology, $HL(M)$]\label{hfu2}
Suppose $M$ is a rank 2, degree 1, decorated cobordism, whose two boundary components are each diffeomorphic to the genus 1 surface $X = T^2$. Let $\operatorname{pt}\in \mc{FR}(X)$ denote the object in $\mc{FR}(X)$ arising from the unique Lagrangian submanifold of $M(X) = \operatorname{pt}$. Then we define the \textbf{Lagrangian $U(2)$ Floer homology HL(M)} to be the abelian group given by the following Hom set:
\begin{align}
HL(M) := \operatorname{Hom}(\mc{FR}(M)\operatorname{pt},\operatorname{pt})
\end{align}
\end{defn}

For us, the important property of the group $HL(M)$ is its similarity to the Lagrangian Floer homology invariants of closed 3-manifolds defined via a Heegaard splitting. Indeed, we can see the relationship more directly as follows. Let $M$ be a (decorated) cobordism from the torus to itself, as in Definition \ref{hfu2}, and let $M_{\pm}$ be a splitting of $M$ by compression bodies, so that each of $M_{\pm}$ is diffeomorphic to a compression body $Y_g$ as in the previous section, with $\partial Y_g = T^2 \cup \Sigma_g$. By Theorem \ref{equal}, $L(M_{\pm}) = \L(M_{\pm}) \subset \mc{R}_g$ is actually a (smooth) Lagrangian submanifold, therefore $\mc{M}(M_{\pm}) = L(M_{\pm})$, and furthermore, the definition of quilted Floer homology reduces to the standard definition. Thus, we have
\begin{lemma}\label{quiltheegaard}
$HL(M) = HF(\L(M_+),\L(M_-)).$
\end{lemma}

\section{Symplectic instanton homology and $(1,1)$ knots}\label{chpt2mainsection}

Now we turn to the original contribution (and main result) of this chapter, where we will use the Lagrangian $U(2)$ Floer homology from Definition \ref{hfu2} to directly define new knot invariants from the symplectic geometry of representation varieties. We will then use a recent (and difficult) result of Ivan Smith to show that these invariants have the same rank as the Heegaard Floer knot homology groups $\widehat{HFK}$ for (1,1) knots in $S^3$ (a simple but interesting class of knots in $S^3$, including all torus knots; see below for a succinct definition). Each of these results requires essentially no new work, beyond that of Smith and Wehrheim-Woodward, and therefore the exposition and proofs of these results will be fairly short.

Let $K$ be a knot in a closed, oriented 3-manifold $M$, and consider the ``sutured manifold'' $(M\backslash \text{tb}(K),s_1,s_2)$ consisting of the complement $M\backslash \text{tb}(K)$ of a tubular neighborhood of $K$, and two sutures $s_1,s_2\subset \partial (M\backslash \text{tb}(K))$, each homeomorphic to an annulus, and with opposite orientation. Attach a thickened annulus $A\times [0,1]$ to $M\backslash \text{tb}(K)$ by gluing $A\times \{0\}$ and $A\times \{1\}$ to $s_1$ and $s_2$, respectively. The resulting manifold, which we call the \textbf{knot closure} $M_K$, and which is canonically associated to $K$, can be viewed as a cobordism $$T^2\stackrel{M_K}\longrightarrow T^2$$ between its two boundary tori. 

Let $E$ denote an appropriate choice of non-trivial $U(2)$ bundle data on $M_K$ to make it into a rank 2, degree 1 decorated cobordism (see Definition \ref{decoratedcobordism}), which we denote by $M_K^E$. When $Y=S^3$, there is a unique choice of $E$ up to diffeomorphism (essentially by Alexander duality, since the diffeomorphism types of bundles in this case is determined by characteristic class data, i.e. by the cohomology), and we denote the corresponding decorated cobordism by $S^3_K$. We can apply the Floer homology group invariant from Definition \ref{hfu2} directly to these decorated cobordisms:

\begin{defn}[symplectic instanton knot homology]\label{maindef1} For a knot $K\subset S^3$, define 
$$\mc{S}(K) := HL(S_K^3),$$ 
the \textbf{symplectic instanton homology} of $K$.

More generally, for a knot $K\subset Y$ and bundle data $E$, define 
$$\mc{S}(K,E) := HL(M_K^E),$$ 
the \textbf{symplectic instanton homology} of the pair $(K,E)$. 
\end{defn}

In each case, $\mc{S}(K)$ is naturally a finitely generated $\Z$-module. Furthermore, the relative spin structures discussed in the preceding sections actually provide $\mc{S}(K)$ with a relative $\Z/4 \Z$ grading, but we will not discuss or investigate this grading any further (in this thesis). 

Our goal now is to prove:

\begin{theorem}[main theorem, chapter 2]\label{main1} For all $(1,1)$ knots $K\subset S^3$, the ranks of $\mc{S}(K)\otimes \C$ and $\HFK(K)\otimes \C$ are equal. 
\end{theorem}

We note that it should be straightforward to extend this result to cover all $(1,1)$ knots in lens spaces, but we will not pursue this case here. 

\subsection{Proof of the main theorem}

In this section we restrict to the case that $M=M_K$ for a knot $K\subset S^3$. 

Suppose $\Sigma'_g$ is a \textbf{doubly-pointed Heegaard surface for $K$}, defined to be a Heegaard surface (see Definition \ref{heegaard}) for $S^3$ which intersects $K$ transversely in two points, and splits it into two unknotted arcs. Let $\Sigma''_g$ denote the intersection of $\Sigma'_g$ and the knot complement $M\backslash \text{tb}(K)$, so that it is homeomorphic to a genus $g$ surface with two disks removed. We can arrange for the sutures $s_1,s_2 \subset \partial (M\backslash \text{tb}(K))$ (from the definition of $M_K$) to be given by the intersection of $\partial (M\backslash \text{tb}(K))$ with a small collar neighborhood of $\Sigma''_g$. 

Recall that $M_K$ is constructed by partially gluing $A\times [0,1]$ to $M\backslash \text{tb}(K)$ along these sutures. Writing the annulus $A$ as $S^1\times [0,1]$, take a center circle $$\beta=S^1\times \{1/2\} \subset A = S^1\times [0,1],$$ and construct a closed genus $g+1$ surface $\Sigma_{g+1}$ in $M_K$ by gluing $\beta\times [0,1]\subset A\times [0,1]$ to $\Sigma''_g$ along its boundary (i.e., attaching a handle). Then we have

\begin{proposition}\label{heegaardequal}
If $\Sigma'_g$ is a Heegaard surface for $K$, then $\Sigma_{g+1}$ splits $M_K$ into two compression bodies $M_\pm$, with $\partial M_\pm = T^2 \cup \Sigma_{g+1}$. We will call such a $\Sigma_{g+1}$ a \textbf{compression Heegaard surface} for $M_K$.
\end{proposition} 

\begin{proof} Let $f: S^3 \to [0,1]$ be a Morse function whose level set $f^{-1}(1/2)$ is equal to $\Sigma'_g\subset S^3$; therefore, $f^{-1}([0,1/2])$ and $f^{-1}([1/2,1])$ are each handlebodies. The restriction of this Morse function to $M\backslash \text{tb}(K)$ can be extended to $M_\pm$ by taking the function on $A\times [0,1] = S^1 \times [0,1] \times [0,1]$ given by projection onto the second factor. Therefore the critical points are unchanged, proving that $M_\pm$ are compression bodies. 
\end{proof}

The group $\HFK(K)$ is constructed from a doubly-pointed Heegaard \textbf{\textit{diagram}} on $\Sigma'_g$. By definition, this is a tuple $(\Sigma'_g ,\{\alpha_i\}, \{\beta_i\},z,w)$, where $\{z,w\}=K\cap \Sigma'_g$ are basepoints in $\Sigma'_g$, and $\{\alpha_i\}$ and $\{\beta_i\}$ are attaching curves for the handlebodies on either side of $\Sigma'_g$ (as in the proof of Lemma \ref{ell}). In this context, $\{\alpha_i\}$ and $\{\beta_i\}$ are called \textbf{Heegaard curves} for the knot $K$ (with respect to the Heegaard surface $\Sigma'_g$). The Heegaard curves are also required to be disjoint from $z$ and $w$; therefore, they naturally restrict to the complement in $\Sigma'_g$ of small neighborhoods of $z$ and $w$. Furthermore, the proof of Proposition \ref{heegaardequal} shows that the curves $\{\alpha_i\}, \{\beta_i\}$ are also attaching curves in $\Sigma_{g+1}\subset M_K$, for the two compression bodies $M_{\pm}$. 

Expanding on Lemma \ref{attach}, for any set of simple closed curves $\{\alpha_i\}\subset \Sigma_{g+1}$, let $\L(\{\alpha_i\})\subset \mc{R}_{g+1}$ denote those representations in $\mc{R}_{g+1}$ which send the conjugacy class of each $\alpha_i$ to $\mathbb{I}\subset SU(2)$. Then, combining the results at the end of $\S \ref{modsu2section}$ with Theorem \ref{quiltheegaard}, we find that 

\begin{theorem}\label{skab} Fix a knot $K\subset S^3$. Then, for any Heegaard curves $\{\alpha_i\}$ and $\{\beta_i\}$ for $K$ (with respect to some doubly-pointed Heegaard diagram), we have 
$$ \mc{S}(K) \cong HF(\L(\{\alpha_i\}),\L(\{\beta_i\}))$$
\end{theorem}
%
%
%
%

We restrict now to the case where $K$ is a (1,1) knot. By definition, this means we can choose a genus 1 doubly-pointed Heegaard diagram $(\Sigma'_1 ,\alpha, \beta,z,w)$ for $K$ (note that in this case, we have only one $\alpha$ and one $\beta$ curve). Let $(\Sigma_2 ,\alpha,\beta)$ be the corresponding compression Heegaard diagram for $M_K$ (generalizing the terminology from Proposition \ref{heegaardequal}). If we choose an area form on $\Sigma_2$, it becomes a symplectic manifold, and since $\alpha$ and $\beta$ are smooth, simple closed curves, they give embedded Lagrangian submanifolds in $\Sigma_2$. Furthermore, it is proved in \cite{abouzaid2008fukaya}, for example, that the Lagrangian Floer homology group $HF(\alpha,\beta)$ is well-defined, and that its rank is equal to the minimal geometric intersection number between any curves isotopic to $\alpha$ and $\beta$ in $\Sigma_2$. This intersection number agrees with the geometric intersection number of $\alpha$ and $\beta$ when viewed as curves in $\Sigma'_1\backslash\{z,w\}$, which is known to give the rank of $\HFK(K)$ (see \cite[$\S 6$]{ozsvath2004holomorphicKnotInvariants}). We conclude that 
$$\HFK(K) \cong HF(\alpha,\beta)$$
Therefore, combined with Theorem \ref{skab}, the following theorem implies Theorem \ref{main1} (Smith's result is where all of the deep mathematics lies):

\begin{theorem}[Smith, \cite{smith2010floer}] $HF(\alpha,\beta)$ is isomorphic to $HF(\L(\{\alpha\}),\L(\{\beta\}))$. Note that Smith's result concerns Floer homology groups defined over $\C$, so the theorem only holds with $\C$ coefficients. 

\end{theorem}

\begin{proof} As a corollary of Theorem 1.1 in \cite{smith2010floer}, there is ``a $\C$-linear equivalence of $\Z/2\Z$-graded split-closed triangulated categories'' 

\begin{equation}\label{equiv}\mc{Y}:D^\pi \mc{F}(\Sigma_2)\cong D^\pi \mc{F}(\mc{R}_2;0) \end{equation}

However, the precise definition of these terms will not be necessary for the proof of the theorem. Above, $\mc{F}(\Sigma_2)$ denotes the balanced Fukaya category of $\Sigma_2$, and $\mc{F}(\mc{R}_2;0) $ denotes a certain orthogonal summand of the monotone Fukaya category of $\mc{R}_2$. For an $A_\infty$ category $\mc{C}$, such as $\mc{F}(\Sigma_2)$ or $\mc{F}(\mc{R}_2;0) $, $D^\pi\mc{C}$ denotes the ``cohomological category $H(T w^\pi \mc{C})$ underlying the split-closure of the category of twisted complexes of $\mc{C}$.'' 

The important thing for us is simply the following. As shown by Smith in \cite{smith2010floer}, the Lagrangians $\alpha,\beta\subset \Sigma_2$ give rise to objects $$[\alpha],[\beta]\in \opn{Ob}(D^\pi \mc{F}(\Sigma_2)),$$ and the Lagrangians $\L(\{\alpha\}),\L(\{\beta\}) \subset \mc{R}_2$ give rise to objects $$[\L(\{\alpha\})],[\L(\{\beta\})]\in \opn{Ob}(D^\pi \mc{F}(\mc{R}_2;0)).$$ Furthermore, $\opn{Hom}([\alpha],[\beta])$ is a chain complex whose homology is isomorphic to $HF(\alpha,\beta)$, and $\opn{Hom}([\L(\{\alpha\})],[\L(\{\beta\})])$ is a chain complex whose homology is isomorphic to $HF(\L(\{\alpha\}),\L(\{\beta\}))$. Finally, for any simple closed curve $\alpha\subset \Sigma_2$, Smith shows that $$\mc{Y}([\alpha]) = [\L(\{\alpha\})] \in  \opn{Ob}(D^\pi \mc{F}(\mc{R}_2;0))$$ Thus we have $\C$-linear maps $$ \mc{Y}: \opn{Hom}([\alpha],[\beta]) \to \opn{Hom}([\L(\{\alpha\})],[\L(\{\beta\})])$$ Since $\mc{Y}$ is an equivalence of categories, these maps are chain homotopy equivalences, implying the theorem. \end{proof}

\begin{rem}
In \cite{smith2010floer}, Smith already observes that his result proves the simplest non-trivial case (that is, genus = 2) of a 2-dimensional generalization of Seiberg-Witten = Donaldson. But, as he points out, the only closed 3-manifolds to which this special case can be applied (in order to identify Heegaard Floer and symplectic instanton invariants) are those admitting a genus 1 Heegaard splitting, i.e., lens spaces, for which the invariants are somewhat trivial. Our contribution is simply to shift focus from closed 3-manifolds to knots. 
\end{rem}

\chapter{A topological formula for the Jones polynomial of two-bridge knots}

In this chapter, we prove a somewhat strange result, relating the Jones polynomial of 2-bridge knots to two sets of geometric invariants, the Atiyah-Patodi-Singer $\rho$-invariants and Ozsv\'ath-Szab\'o $d$-invariants. In fact, we believe this result should be interpretable in terms of an additional $\Z$-grading on the symplectic instanton knot homology $\mc{S}(K)$ defined in the preceding chapter. This interpretation is motivated by the relation between $\mc{R}_g$, the moduli space of twisted $SU(2)$ representations on a genus $g$ surface (and one of the key ingredients for constructing $\mc{S}(K)$), and Chern-Simons theory (and therefore the Jones polynomial). We have developed a strategy to find this grading, using ideas similar to those in \cite{jeffrey1993half}, but it has not yet been implemented. 

\section{Formula for the Jones polynomial of 2-bridge knots} Let $K=K(p,q)$ be a 2-bridge link, with branched double cover the lens space $L(p,q)$. Let $\opn{Spin}^c(p,q)$ denote the set of conjugacy classes of $\opn{Spin}^{c}$ structures on $L(p,q)$, and let $\mathcal{M}(p,q)$ denote the set of conjugacy classes of $SU(2)$ representations of $\pi_{1}(L(p,q))$ (alternatively, we write $\opn{Spin}^c(K)$ for $\opn{Spin}^c(p,q)$ if $p$ and $q$ are unspecified). In what follows, we will construct a certain map $\iota: \opn{Spin}^c(p,q) \to \mathcal{M}(p,q)$. Let $$I(\mfs):=8d(\mfs)- \rho(\iota(\mfs))$$ where $\mfs \in \opn{Spin}^c(p,q)$, $d(\mathfrak{s})\in \Q$ denotes the Ozsv\'ath-Szabo $d$-invariant, and $\rho(\alpha)\in\Q$ denotes the Atiyah-Patodi-Singer $\rho$-invariant of a representation $\alpha$. Note that because $\pi_{1}(L(p,q))$ is abelian, $\mathcal{M}(p,q)$ is the same as the set of characters of $H_{1}(L(p,q))$ up to conjugation, and $\rho(\alpha)$ agrees with the corresponding Casson-Gordon knot invariant. It turns out that $I(\mfs)$ is an integer, for all $\mfs$. 

Fix an orientation on $K$, and let $\mathcal{O}$ be the set of relative orientations for the underlying unoriented link. If $K$ is a knot, then set $K^{o_{1}}=K$ for the unique $o_{1}\in\mathcal{O}$. If $K$ is a 2-component link, then we set $K^{o_{1}}=K$ and $K^{o_{2}}$ is $K$ with the orientation on one component reversed (it will not matter which component we choose, as we only care about this new link through its signature, which is unaffected by this choice). Indeed, either way we have$$\sigma(K^{o_{2}})=\sigma(K)+2lk(K)$$ where $\sigma(K)$ is the knot signature and $lk$ is the linking number

Let $J(K)$ denote the Jones polynomial of $K$ (to fix convention, we mean the particular Jones polynomial which is the graded Euler characteristic of reduced Khovanov homology). Our purpose here is to prove the following theorem (note that in our notation, $i=\sqrt{-1}$):

\begin{theorem}[main theorem, chapter 3]\label{main2} 
$$i^{-\sigma(K)}q^{3\sigma(K)}J(K)= \sum_{o\in\mathcal{O}}(iq)^{2\sigma(K^{o})} +\left(q^{-1}-q^{1}\right)\sum_{\mfs\in\opn{Spin}^c(K)}(iq)^{I(\mfs)}$$
\end{theorem}

(This is just one of many possible ways to correctly encode the signs). 

It is possible that Theorem \ref{main2} might be useful for calculating the 4-ball genus of some 2-bridge knots. We hope to explore applications, as well as a possible generalization to all alternating knots, in a subsequent paper. 

\section{Khovanov homology and skein relations for sets of coefficients}\label{skeinsection}

Here we introduce the relevant information about the Jones polynomial, which we approach via (unreduced) Khovanov homology. Fix a non-split, prime, alternating, oriented knot or 2-component link $K$, and let $\sigma(K)$ denote the signature of $K$. For $a\in \Z$, define two abstract bigraded $\Q$ vector spaces,  $$E(a):=\Q^{\frac{a+\sigma}{2},a-1}\oplus\Q^{\frac{a+\sigma}{2},a+1},$$ $$ K(a):=\Q^{\frac{a+\sigma-1}{2},a-2}\oplus \Q^{\frac{a+\sigma+1}{2},a+2},$$ which look the following way when depicted in the $(i,j)$ plane:
\begin{figure}[htp]
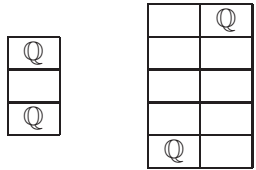

\centering
\begin{tabular}{|c|}\hline $\Q$ \\\hline  \\\hline $\Q$ \\\hline \end{tabular}
~~~~~~~~
\begin{tabular}{|c|c|}\hline  & $\Q$ \\\hline  &  \\\hline  &  \\\hline  &  \\\hline $\Q$ &  \\\hline \end{tabular}
\caption{$E(a)$ and $K(a)$, where $a$ is the grading of the ``middle'' horizontal row.}\label{F4}
\end{figure}

The following is a reinterpretation of a theorem of Lee from \cite{lee2005endomorphism}. Namely, when $K$ has $n$ components, there are $\frac{det(K)-2^{n}}{2}$ well-defined integers $d_{i}$ ($det(K)$ being the knot determinant of $K$), such that $$Kh(K,\Q)= \bigoplus_{o\in\mathcal{O}}E(2\sigma(K^{o})-3\sigma(K))\oplus \bigoplus_{i=1}^{\frac{det(K)-2^{n}}{2}} K(d_{i})$$ as bigraded vector spaces (we use the same notation for relative orientations as in the introduction).

Therefore, for an alternating knot $K$, we can combine all the grading information of Khovanov homology into a set of integers $$M(K):=\{c_{1},\dots,c_{|\mathcal{O}|}\} \cup \{d_{1},\dots, d_{\frac{det(K)-2^{n}}{2}}\}$$ where $c_{i}=2\sigma(K^{o_{i}})-3\sigma(K)$ (again, recall that our notational convention is that $o_{1}$ agrees with the orientation on $K$, so that $c_{1}=-\sigma(K)$, and when $K$ has 2 components, $c_{2}=4lk-\sigma(K)$). 

Just like the Jones polynomial, and Khovanov homology itself, this set satisfies a simple skein relation. Suppose first that we have a two component, non-split, alternating oriented diagram $K$, and suppose we choose a positive crossing between the two components, to smooth into two new diagrams $K_{0}$ and $K_{1}$, both of which will be knots. Our conventions dictate that $K_{0}$ is the diagram which inherits its orientation from $K$, and we let $$e=n_{-}(K_{1})-n_{-}(K)$$ (the difference in the number of negative crossings in the two diagrams), where $K_{1}$ is given an arbitrary orientation. Then, we have 

\begin{equation*}\label{a0}Kh^{i,j}_{K}= Kh_{K_{0}}^{i,j-1} \oplus  Kh^{i-e-1,j-3e-2}_{K_{1}}\end{equation*}

which tells us that $$M(K)=\left\{M(K_{0})+1\right\} \cup \left\{M(K_{1})+3e+2\right\}$$

with $c_{1}(K)=c_{1}(K_{0})+1$ and $c_{2}(K)=c_{1}(K_{1})+3e+2$.

Things are a bit more complicated when we begin with a knot $K$. Assume again that we start with a positive crossing, so that $K_{0}$ inherits the orientation and is therefore a link. In this case all but one of the elements of $M(K)$ are given by the set $$\left\{M(K_{0})+1\right\}\backslash \{c_{2}(K_{0})+1\}\cup \left\{M(K_{1})+3e+2\right\}\backslash\{c_{1}(K_{1})+3e+2\},$$ in particular we have $c_{1}(K)=c_{1}(K_{0})+1$. For the final element of $M(K)$, note that a simple computation tells us that $e=2lk(K_{0})$, and so by the signature formulas from \cite{manolescu2005concordance} follows that $$\left(c_{1}(K_{1})+3e+2\right)-\left(c_{2}(K_{0})+1\right)=2$$ One can then check that the remaining element of $M(K)$ is given by $c_{1}(K_{1})+3e+1=c_{2}(K_{0})+2$. 


As with the Jones polynomial, a simple inductive argument shows that these skein relations determine the sets $M(K)$, along with their value on the unknot, which is just $\{0\}$. 


%


\section{$d$-invariants and $\rho$-invariants}

We now specialize to the case of 2-bridge links $K(p,q)$. Recall that $L(p,q)$ denote the lens space which is the branched double cover of $K(p,q)$, and $\opn{Spin}^c(p,q)$ and $\mathcal{M}(p,q)$ denote the set of conjugacy classes of spin$^{c}$ structures and $SU(2)$ representatations for $L(p,q)$, respectively. Also, we will define the map $\iota: \opn{Spin}^c(p,q)\to\mathcal{M}(p,q)$, and we have $$I(\mfs):=8d(\mfs)+ \rho(\iota(\mfs))$$ In the notation of the previous section, the theorem we will prove is

\begin{theorem}\label{skeinProof} For a 2-bridge link $K(p,q),$ \begin{equation}\label{main}M(K(p,q))= \{-I(\mfs_{i})-3\sigma(K)\}_{\mfs_{i}\in \opn{Spin}^c(p,q)}\end{equation}   And, more specifically, we have \begin{equation}\label{spin}c_{i}(K)=-I(\mfs_{i})-3\sigma(K)\end{equation} where $\mfs_{1}$ is the unique spin structure on $L(p,q)$ if $K(p,q)$ is a knot, and $\mfs_{1}$ and $\mfs_{2}$ are the two spin structures on $L(p,q)$ when $K(p,q)$ is a link, in which case we take $\mfs_{1}$ to be the unique spin structure which is induced by the canonical spin structure on $S^{3}$. 
\end{theorem}

\begin{proof}

We mention that it will turn out that $\iota$ sends the spin structures to the trivial representation, whose $\rho$-invariant is 0. Therefore Theorem \ref{skeinProof} implies that the $d$-invariant of the spin structure is $-1/4$ times the knot signature, which is already known. 



The theorem in general is proved by showing that the set on the right hand side of (\ref{main}) satisfies the same skein relation as $M(K)$. Begin with the two bridge link diagram $K=K(p,q)$, with $p>q>0$, assumed in canonical 2-bridge form (see \cite{riley1984nonabelian}), and assume that the top crossing is positive. By \cite{riley1984nonabelian}, $$K_{0}=-K(q,p)=K(q,-p),~K_{1}=K(p-q,q)$$ (the minus sign denotes mirror image). 


The proof consists of two steps, considering first $K_0$, and then $K_1$. 
\newline
\begin{enumerate}
\item[Step 1] $\mathbf{K_{0}=K(q,-p)}$. 

By \cite{ozsvath2003absolutely}, there is an indexing of the sets $\opn{Spin}^c(p,-q)$ and $\opn{Spin}^c(q,-p)$ by $\Z/p\Z$ and $\Z/q\Z$, respectively, such that for all $0\leq i\leq p+q$,  $$d(-L(p,q),i)=\frac{1}{4}-\frac{\left(2i+1-p-q\right)^{2}}{4pq} - d(-L(q,p),j)$$ where $j\equiv i$ mod $q$. We can rewrite this as   \begin{equation}\label{d}d(L(p,q),i)-d(L(q,-p),i)=-\frac{1}{4}+\frac{\left(2i+1-p-q\right)^{2}}{4pq}\end{equation} because the $d$ invariants change sign for mirror images. Now, let $\mathcal{P}$ be the set of integers $$i\in \left[\frac{p-1}{2},\frac{p-1+2q}{2}\right]$$ 
Note that reduction mod $q$ gives a bijection $\mathcal{P}\cong \Z/q\Z$. Now, we will need to choose positive integers $r ,s>0\in\Z $ such that 
\begin{equation}\label{rep}ps-qr=1\end{equation}
There is also a natural way to index the $SU(2)$ representations of $L(p,r)\cong -L(p,q)$ and $L(q,s)\cong L(q,p)$ by $\Z/p\Z$ and $\Z/q\Z$ respectiely, see \cite{casson1986cobordism} (where it is done for characters rather than $SU(2)$ representations). In terms of this indexing, for $n\in \Z/a\Z$, for example, we have the formula $$\rho(a,b,n)=4\left(\text{area}\Delta\left(n,n\frac{b}{a}\right)-\text{int}\Delta\left( n,n\frac{b}{a}\right)\right)$$ 

Here, $\Delta(a,b)$ is the triangle with vertices at $(0,0), (a,0)$ and $(a,b)$. ``Int'' is the number of integer points in $\Delta$, without counting $(0,0)$, and with boundary points counting $1/2$, except for vertices which count for $1/4$. 

We are now ready to compute the skein relation. For each $i\in\mathcal{P}$, let $$n(i)=2i+1-p-q,$$ and use this to define a representation of $\mathcal{M}(q,s)$ via the above indexing. This gives the map $$\iota :\opn{Spin}^c(q,-p)\to \mathcal{M}(q,s)\cong \mathcal{M}(q,-p
)$$ mentioned in the introduction. In the same way, by reducing the elements of $\mathcal{P}$ mod $p$, we get a map from a \textit{subset} of $\opn{Spin}^c(p,q)$ to $\mathcal{M}(p,r)$. 

We now have two sub-cases, depending on the parity of $p$. 

\begin{enumerate}
\item \textbf{$\mathbf{p}$ is even ($\mathbf{K(p,q)}$ is a link)}

Recall that we are using $\mathcal{P}$ to index corresponding elements of $\opn{Spin}^c(p,q)$ and $\opn{Spin}^c(q,-p)$. Also, the numbers $n(i)$, for $i\in\mathcal{P}$, have the form \begin{equation}\label{range}\{-(q-1), -(q-3),\dots 0,\dots q-3,q-1\}\end{equation} and we use these to represent corresponding elements of $\mathcal{M}(p,r)$ and $\mathcal{M}(q,s)$. Note that $\rho(p,q,n(i))=\rho(p,q,-n(i))$, so we can restrict to $n(i)>0$. We prove:

\begin{lemma} For $i \in \mathcal{P}$, \begin{equation}\label{rho}\rho(p,r,n(i))-\rho(q,s,n(i))=-2\frac{n(i)^{2}}{pq}\end{equation} 

\end{lemma}

\begin{proof}

It follows from (\ref{rep}) that the above difference is just the difference of the area terms, and therefore the claim amounts to proving that the int terms are equal. It follows from (\ref{rep}) and (\ref{range}) that $$\frac{p-2}{2}\leq i\leq q+\frac{p-2}{2}$$ and therefore $$|n| \leq q-1 <q$$ Therefore we have $$|\frac{ns}{p}-\frac{nr}{q}|\leq \frac{1}{p}$$ We are considering the two line segments $(pt,rt)$ for $t\in[0,n/p]$ and $(qt,st)$ for $t\in[0,n/q]$. Any segment of the line $x=m$ which lies between these two lines and has $|m|<|n|$, will have length less than or equal to $|\frac{ns}{p}-\frac{nr}{q}|\leq \frac{1}{2p}$ (the length at $x=n$). But also, the $y$ values at $x=m$ of each line will be non-integer rational numbers with denominator $p$ and $q$, respectively. Therefore, if they lie on either side of an integer, the total length of that segment must be $>\frac{1}{p}+\frac{1}{q}>\frac{1}{2p}$, a contradiction. This implies that the triangles contain the same number of lattice points. \end{proof}

Therefore, taking orientations into account and combining equations (\ref{d}) and (\ref{rho}), we have  \begin{align}\label{s1}I(\mfs_{i}(K))-I(\mfs_{i}(K_{0}))& = \\ 
8d(p,q,i)-\rho(p,q,n(i)) - \left( 8d(q,-p,i)-\rho(q,-p,n(i)) \right)& = \notag\\
& = -2 \notag\end{align} Now, because we started with a positive crossing, we use [OR] to see that $$-3\sigma(K)-(-3\sigma(K_{0}))=3$$ Adding 3 to -2 to get 1 completes the proof in this case.

\item \textbf{$\mathbf{p}$ is odd}

Suppose now that $p$ is odd. The proof goes through exactly as before, except that now, the image of $n(i)$ for $i\in\mathcal{P}$ is either 

\begin{equation}\label{range2}  \{-q,-(q-2),\dots 1,\dots q+2,q\}   \end{equation}

when $q$ is odd, or 

\begin{equation}\label{range2}  \{-q,-(q-2),\dots 0,\dots q+2,q\}   \end{equation}

when $q$ is even. In either case, we can no longer assume that $n(i)$ is strictly less than $q$. Indeed, consider the case that $n(i)=q$. Then equation (\ref{rep}) tells us that we have one more lattice point in the triangle $\Delta(q,s)$ than in the triangle $\Delta(q,qr/p)$, and it is a vertex. Going through the equations, we get the new formula 

 \begin{equation}\label{rho2}\rho(p,r,q)-\rho(q,s,q)=-2\frac{n(i)^{2}}{pq}+1\end{equation} 

We therefore conclude that $$I(\mfs_{q}(K))-I(\mfs_{q}(K_{1}))= -2+1=-1$$ Again, comparing to the skein relations for $M(K)$ from $\S \ref{skeinsection}$, this complete the proof in this case.

\end{enumerate}

\item[Step 2] $\mathbf{K_{1}=K(p-q,q)}$

Setting $q'=p-q $ using the homeomorphisms $L(p,q)\cong - L(p,q')$ and $L(p-q,q)\cong L(q',p)$ allows Step 2 to be treated as Step 1. This defines the map $\iota$ on a new set $\mathcal{P'}$ of spin$^{c}$ structures on $L(p,q)$, and one must check that this new set is disjoint from the previous (and therefore together, they make up all of $\opn{Spin}^c(p,q)$. Furthermore, one must check that the map $\iota$ is consistent across all the skein relations. It is worthwhile, though not strictly necessary, to also prove that the particular elements indexed either by $0$ or $q$ correspond to the spin structures. 

\end{enumerate}

Finally, we must consider the case where the top crossing of $K$, in its canonical 2-bridge form, is negative instead of positive. However, we can simply take the mirror image, noting that the $\rho$ and $d$-invariants switch sign under this operation. 
\end{proof}

\chapter{Strong L-spaces and left-orderability}

In this chapter, we discuss a relationship between Heegaard Floer homology of a three manifold $Y$, and certain orderability properties of its fundamental group. This is joint work with Adam Levine. We note that we have taken practically everything in this chapter, including the introductory exposition, directly from the published paper \cite{levine2012strong}. 

Heegaard Floer homology has been an extremely effective tool for answering
classical questions about 3-manifolds, particularly concerning the genera of
embedded surfaces in particular homology classes \cite{ozsvath2004holomorphicGenus}. However,
surprisingly little is known about the relationship between Heegaard Floer
homology and topological properties of Heegaard splittings of $3$-manifolds,
even though a Heegaard diagram is an essential ingredient in defining the
Heegaard Floer homology of a closed $3$-manifold $Y$. In particular, a Heegaard
diagram provides a presentation of the fundamental group of $Y$, and it is
natural to ask how this presentation is related to the Heegaard Floer chain
complex. In this paper, we shall investigate one such connection.

A \textbf{left-ordering} on a non-trivial group $G$ is a total order $<$ on the
elements of $G$ such that $g < h$ implies $kg < kh$ for any $g,h,k \in G$. A
group $G$ is called \textbf{left-orderable} if it is nontrivial and admits at
least one left-ordering. The question of which $3$-manifolds have
left-orderable fundamental group has been of considerable interest and is
closely connected to the study of foliations. For instance, if $Y$ admits an
$\R$-covered foliation (i.e., a taut foliation such that the leaf-space of the
induced foliation on the universal cover $\widetilde{Y}$ is homeomorphic to
$\R$), then $\pi_1(Y)$ is left-orderable. Howie and Short showed that the fundamental group of any irreducible $3$-manifold $Y$ with
$b_1(Y)>0$ is left-orderable, reducing the question to that of rational
homology spheres.

In its simplest form, Heegaard Floer homology associates to a closed, oriented
$3$-manifold $Y$ a $\Z/2\Z$--graded, finitely generated abelian group $\HF(Y)$.
This group is computed as the homology of a free chain complex $\CF(\HH)$
associated to a Heegaard diagram $\HH$ for $Y$; different choices of diagrams
for the same manifold yield chain-homotopy-equivalent complexes. The group
$\CF(\HH)$ depends only on the combinatorics of $\HH$, but the differential on
$\CF(\HH)$ involves counts of holomorphic curves that rely on auxiliary choices
of analytic data. If $Y$ is a rational homology sphere, then the Euler
characteristic of $\HF(Y)$ is equal to $\abs{H_1(Y;\Z)}$, which implies that
the rank of $\HF(Y)$ is greater than $\abs{H_1(Y;\Z)}$. $Y$ is called an
\textbf{$L$-space} if $\HF(Y) \cong \Z^{\abs{H_1(Y;\Z)}}$; thus, $L$-spaces have
the simplest possible Heegaard Floer homology. Examples of $L$-spaces include
$S^3$, lens spaces (whence the name), all manifolds with finite fundamental
group, and double branched covers of alternating (or, more broadly,
\textbf{quasi-alternating}) links. Additionally, Ozsv\'ath and Szab\'o
\cite{ozsvath2004holomorphicGenus} showed that if $Y$ is an $L$-space, it does not admit any taut
foliation; whether the converse is true is an open question.

The following related conjecture, stated formally by Boyer, Gordon, and Watson
\cite{boyer2013spaces}, has recently been the subject of considerable
attention:
\begin{conj} \label{conj:Lspace}
Let $Y$ be a closed, connected, 3-manifold. Then $\pi_{1}(Y)$ is not
left-orderable if and only if $Y$ is an $L$-space.
\end{conj}
This conjecture is known to hold for all geometric, non-hyperbolic 3-manifolds
\cite{boyer2013spaces, boyer2005orderable, lisca2007ozsvath,
peters2009spaces}. Additionally, Boyer--Gordon--Watson \cite{boyer2013spaces}
and Greene \cite{greene2011alternating} have shown that the double branched cover
of any non-split alternating link in $S^3$
--- which is generically a hyperbolic $3$-manifold --- has non-left-orderable
fundamental group.

Here, we prove Conjecture \ref{conj:Lspace} for manifolds that are
``$L$-spaces on the chain level.'' To be precise, we call a 3-manifold $Y$ a
\textbf{strong $L$-space} if it admits a Heegaard diagram $\HH$ such that
$\CF(\HH) \cong \Z^{\abs{H_1(Y;\Z)}}$. This purely combinatorial condition
implies that the differential on $\CF(\HH)$ vanishes, without any consideration
of holomorphic disks. We call such a Heegaard diagram a \textbf{strong Heegaard
diagram}. By considering the presentation for $\pi_1(Y)$ associated to a strong
Heegaard diagram, we prove:
\begin{theorem}[main theorem, chapter 4]\label{thm:main}
If $Y$ is a strong $L$-space, then $\pi_1(Y)$ is not left-orderable.
\end{theorem}

The standard Heegaard diagram for a lens space is easily seen to be a strong
diagram. Moreover, Greene \cite{greene2013spanning} constructed a strong Heegaard
diagram for the double branched cover of any alternating link in $S^3$; indeed,
Boyer--Gordon--Watson's proof that the fundamental group of such a manifold is
not left-orderable makes use of the group presentations associated to that
Heegaard diagram. At present, we do not know of any strong $L$-space that
cannot be realized as the double branched cover of an alternating link; while
it seems unlikely that every strong $L$-space can be realized in this manner,
it is unclear what obstructions could be used to prove this claim. (Indeed, the
question of finding an alternate characterization of alternating links is a
famous open problem posed by R. H. Fox.) Nevertheless, our theorem seems like a
useful step in the direction of Conjecture \ref{conj:Lspace} in that it relies
only on data contained in the Heegaard Floer chain complex.

Furthermore, the following theorem, which is well-known but does not appear in
the literature, indicates that being a strong $L$-space is a fairly restrictive
condition:
\begin{theorem} \label{thm:S3}
If $Y$ is an integer homology sphere that is a strong $L$-space, then $Y \cong
S^3$.
\end{theorem}
In particular, there exist integer homology spheres that are $L$-spaces (e.g.,
the Poincar\'e homology sphere) but not strong $L$-spaces. The fact that the
condition of being a strong $L$-space detects $S^3$ suggests that it might be
possible to obtain a more explicit characterization or even a complete
classification of strong $L$-spaces. We shall present a graph-theoretic
proof of Theorem \ref{thm:S3} due to J. Greene.

\section{Proofs of Theorem \ref{thm:main} and \ref{thm:S3}}

To prove Theorem \ref{thm:main}, we will use a simple obstruction to
left-orderability that can be applied to group presentations.

Let $X$ denote the set of symbols $\{0,+,-,*\}$. These symbols are meant to
represent the possible signs of real numbers: $+$ and $-$ represent positive
and negative numbers, respectively, and $*$ represents a number whose sign is
not known. As such, we define an associative multiplication operation on $X$ by
the following rules: (1) $0 \cdot \epsilon = 0$ for any $\epsilon \in X$; (2)
$+ \cdot + = - \cdot - = +$; (3) $+ \cdot - = - \cdot + = -$; and (4) $\epsilon
\cdot
* =
* \cdot \epsilon =
*$ for $\epsilon \in \{+,-,*\}$.

A group presentation $\GG = \gen{ x_1,\dots, x_m | r_1, \dots, r_n}$ gives rise
to an $m \times n$ matrix $E(\GG) = (\epsilon_{i,j})$ with entries in $X$ by
the following rule:
\begin{equation} \label{eq:epsilonij}
\epsilon_{i,j} = \begin{cases}
 0 & \text{if neither $x_i$ nor $x_i^{-1}$ occur in $r_j$} \\
 + & \text{if $x_i$ appears in $r_j$ but $x_i^{-1}$ does not} \\
 - & \text{if $x_i^{-1}$ appears in $r_j$ but $x_i$ does not} \\
 * & \text{if both $x_i$ and $x_i^{-1}$ occur in $r_j$}. \\
\end{cases}
\end{equation}

\begin{lemma} \label{lemma:notLO}
Let $\GG = \gen{ x_1,\dots, x_m | r_1, \dots, r_n}$ be a group presentation
such that for any $d_1, \dots, d_m \in \{0,+,-\}$, not all zero, the matrix $M$
obtained from $E(\GG)$ by multiplying the $i\Th$ row by $d_i$ has a nonzero
column whose nonzero entries are either all $+$ or all $-$. Then the group $G$
presented by $\GG$ is not left-orderable.
\end{lemma}

\begin{proof}
Suppose that $<$ is a left-ordering on $G$, and let $d_i$ be $0$, $+$, or $-$
according to whether $x_i=1$, $x_i>1$, or $x_i<1$ in $G$. Since $G$ is
nontrivial, at least one of the $d_i$ is nonzero. If the $j\Th$ column of $M$
is nonzero and has entries in $\{0,+\}$ (resp.~$\{0,-\}$), the relator $r_j$ is
a product of generators $x_i$ that are all nonnegative (resp.~nonpositive) in
$G$, and at least one of which is strictly positive (resp.~negative). Thus,
$r_j>1$ (resp.~$r_j<1$) in $G$, which contradicts the fact that $r_j$ is a
relator.
\end{proof}

We shall focus on presentations with the same number of generators as
relations. For a permutation $\sigma \in S_n$, let $\sign(\sigma) \in \{+,-\}$
denote the sign of $\sigma$ ($+$ if $\sigma$ is even, $-$ if $\sigma$ is odd).
The key technical lemma is the following:

\begin{lemma} \label{lemma:matrix}
Let $\GG = \gen{ x_1,\dots, x_n | r_1, \dots, r_n}$ be a group presentation
such that $E(\GG)$ has the following properties:
\begin{enumerate}
\item There exists some permutation $\sigma_0 \in S_n$ such that
$\epsilon_{1,\sigma_0(1)}, \dots, \epsilon_{n,\sigma_0(n)}$ are all nonzero.
\item For any permutation $\sigma \in S_n$ such that
$\epsilon_{1,\sigma(1)}, \dots, \epsilon_{n,\sigma(n)}$ are all nonzero, we
have $\epsilon_{1,\sigma(1)}, \dots, \epsilon_{n,\sigma(n)} \in \{+,-\}$.
\item For any two permutations $\sigma, \sigma'$ as in (2), we have
\[
\sign(\sigma) \cdot \epsilon_{1,\sigma(1)} \cdot \dots \cdot
\epsilon_{n,\sigma(n)} = \sign(\sigma') \cdot \epsilon_{1,\sigma'(1)} \cdot
\dots \cdot \epsilon_{n,\sigma'(n)}.
\]
\end{enumerate}
Then the group $G$ presented by $\GG$ is not left-orderable.
\end{lemma}

In other words, if we consider the formal determinant
\[
\det(E(\GG)) = \sum_{\sigma \in S_n} \sign(\sigma) \cdot \epsilon_{1,\sigma(1)}
\cdot \dots \cdot \epsilon_{n,\sigma(n)},
\]
condition (1) says that at least one summand is nonzero, condition (2) says
that no nonzero summand contains a $*$, and condition (3) says that every
nonzero summand has the same sign.

\begin{proof}
By reordering the generators and relations, it suffices to assume that
$\sigma_0$ from condition (1) is the identity, so that $\epsilon_{i,i} \ne 0$
for $i= 1, \dots, n$, and hence $\epsilon_{i,i} \in \{+,-\}$ by condition (2).
We shall show that $E(\GG)$ satisfies the hypotheses of Lemma
\ref{lemma:notLO}.

Suppose, then, toward a contradiction, that $d_1, \dots, d_n$ are elements of
$\{0,+,-\}$, not all zero, such that every nonzero column of the matrix $M$
obtained as in Lemma \ref{lemma:notLO} contains a nonzero off-diagonal entry
(perhaps a $*$) that is not equal to the diagonal entry in that column. Denote
the $(i,j)\Th$ entry of $M$ by $m_{i,j}$.

We may inductively construct a sequence of distinct indices $i_1, \dots, i_k
\in \{1, \dots, n\}$ such that
\begin{enumerate}
\item[(A)] $m_{i_j,i_j} \in \{+,-\}$ for each $j=1, \dots, m$,
and
\item[(B)]$m_{i_{j+1},i_{j}} \ne 0$ and $m_{i_{j+1},i_j} \ne
m_{i_j,i_j}$
\end{enumerate}
for each $j=1, \dots, k$, taken modulo $k$. Specifically, we begin by choosing
any $i_1$ such that $m_{i_1,i_1} \ne 0$. Given $i_j$, our assumption on $M$
states that we can choose $i_{j+1}$ satisfying assumption (B) above; we then
have $m_{i_{j+1},i_{j+1}} \ne 0$ since otherwise the whole $i_{j+1}\Th$ row
would have to be zero. Repeating this procedure, we eventually obtain an index
$i_k$ that is equal to some previously occurring index $i_{k'}$, where $k'<k$.
The sequence $i_{k'+1}, \dots, i_k$, relabeled accordingly, then satisfies the
assumptions (A) and (B).

Define a $k$-cycle $\sigma\in S_n$ by $\sigma(i_j) = i_{j+1}$ for $j=1, \dots,
k$ mod $k$, and $\sigma(i') = i'$ for $i' \not\in \{i_1, \dots, i_k\}$. By
construction, $\epsilon_{i, \sigma(i)} \ne 0$ for each $i = 1, \dots, n$, so
the sequence $(\epsilon_{1, \sigma(1)}, \dots, \epsilon_{n,\sigma(n)})$
contains no $*$s by condition (2). The sequences $(\epsilon_{1, \sigma(1)},
\dots, \epsilon_{n,\sigma(n)})$ and $(\epsilon_{1,1}, \dots, \epsilon_{n,n})$
differ in exactly $k$ entries, and the signature of $\sigma$ is $(-1)^{k-1}$.
This implies that
\[
\sign(\sigma) \cdot \epsilon_{1, \sigma(1)} \cdot \dots \cdot
\epsilon_{n,\sigma(n)} = (-1)^{2k-1} \sign(\operatorname{id}) \cdot
\epsilon_{1, 1} \cdot \dots \cdot \epsilon_{n,n},
\]
which contradicts condition (3).
\end{proof}

Now we will apply Lemma \ref{lemma:matrix} to prove Theorem \ref{thm:main}. We
first recall some basic facts about the Heegaard Floer chain complex. A
\textbf{Heegaard diagram} is a tuple $\HH = (\Sigma, \bm\alpha, \bm\beta)$, where
$\Sigma$ is a closed, oriented surface of genus $g$, $\bm\alpha = (\alpha_1,
\dots, \alpha_g)$ and $\bm\beta = (\beta_1, \dots, \beta_g)$ are each
$g$-tuples of pairwise disjoint simple closed curves on $\Sigma$ that are
linearly independent in $H_1(\Sigma;\Z)$, and each pair of curves $\alpha_i$
and $\beta_j$ intersect transversely.  A Heegaard diagram $\HH$ determines a
closed, oriented $3$-manifold $Y = Y_\HH$ with a self-indexing Morse function
$f: Y \to [0,3]$ such that $\Sigma = f^{-1}(3/2)$, the $\alpha$ circles are the
belt circles of the $1$-handles of $Y$, and the $\beta$ circles are the
attaching circles of the $2$-handles. If we orient the $\alpha$ and $\beta$
circles, the Heegaard diagram determines a group presentation
\[
\pi_1(Y) = \gen{a_1, \dots, a_g \mid b_1, \dots, b_g},
\]
where the generators $a_1, \dots, a_g$ correspond to the $\alpha$ circles, and
$b_j$ is the word obtained as follows: If $p_1, \dots, p_k$ are the
intersection points of $\beta_j$ with the $\alpha$ curves, indexed according to
the order in which they occur as one traverses $\beta_i$, and $p_\ell \in
\alpha_{i_\ell} \cap \beta_i$ for $\ell = 1, \dots, k$, then
\begin{equation} \label{eq:relation}
b_j = \prod_{\ell = 1}^k a_{i_\ell}^{\eta(p_i)},
\end{equation}
where $\eta(p_i) \in \{\pm 1\}$ is the local intersection number of
$\alpha_{i_\ell}$ and $\beta_j$ at $p_i$.

Let $\Sym^g(\Sigma)$ denote the $g\Th$ symmetric product of $\Sigma$, and let
$\T_\alpha, \T_\beta \subset \Sym^g(\Sigma)$ be the $g$-dimensional tori
$\alpha_1 \times \dots \times \alpha_g$ and $\beta_1 \times \dots \times
\beta_g$, which intersect transversely in a finite number of points. Assuming
$Y$ is a rational homology sphere, $\CF(\HH)$ is the free abelian group
generated by points in $\SS_\HH = \T_\alpha \cap \T_\beta$.\footnote{For
general $3$-manifolds, we must restrict to a particular class of so-called
admissible diagrams.} More explicitly, these are tuples $\x = (x_1, \dots,
x_g)$, where $x_i \in \alpha_i \cap \beta_{\sigma(i)}$ for some permutation
$\sigma \in S_g$. The differential on $\CF(\HH)$ counts holomorphic Whitney
disks connecting points of $\SS_\HH$ (and depends on an additional choice of a
basepoint $z \in \Sigma$), but we do not need to describe this in any detail
here.

Orienting the $\alpha$ and $\beta$ circles determines orientations of
$\T_\alpha$ and $\T_\beta$. For $\x \in \SS_\HH$, let $\eta(\x)$ denote the
local intersection number of $\T_\alpha$ and $\T_\beta$ at $\x$. It is not hard
to see that if $\x = (x_1, \dots, x_g)$ with $x_i \in \alpha_i \cap
\beta_{\sigma(i)}$, we have
\begin{equation} \label{eq:grading}
\eta(\x) = \sign(\sigma) \prod_{i=1}^g \eta(x_i).
\end{equation}
These orientations determine a $\Z/2$-valued grading $\gr$ on $\CF(Y)$ by the
rule that $(-1)^{\gr(\x)} = \eta(\x)$; the differential shifts this grading by
$1$. If $Y$ is a rational homology sphere, then with respect to this grading,
we have $\chi( \CF(\HH)) = \pm \abs{H_1(Y;\Z)}$; we may choose the orientations
such that the sign is positive. (See \cite[Section 5]{ozsvath2004holomorphicDiskInvariantsPropertiesAndApplications} for
further details.)

The proof of Theorem \ref{thm:main} is thus completed with the following:

\begin{lemma}
If $\HH$ is a strong Heegaard diagram for a strong $L$-space $Y$, then the
corresponding presentation for $\pi_1(Y)$ satisfies the hypotheses of Lemma
\ref{lemma:matrix}.
\end{lemma}

\begin{proof}
If $\rank( \CF(\HH)) = \chi(\CF(\HH)) = \abs{H_1(Y;\Z)}$, then $\CF(\HH)$ is
supported in a single grading, so $\eta(\x) = 1$ for all $\x \in \T_\alpha \cap
\T_\beta$. The result then follows quickly from equations \eqref{eq:epsilonij},
\eqref{eq:relation}, and \eqref{eq:grading}. Specifically, since $\SS_\HH \ne
\emptyset$, there exists $\sigma_0 \in S_g$ such that $\alpha_i \cap
\beta_{\sigma_0(i)} \ne \emptyset$ for each $i$, and hence $\epsilon_{i,
\sigma_0(i)} \ne 0$. If $\alpha_i$ and $\beta_j$ contain a point $x$ that is
part of some $\x \in \SS_\HH$, then every other point $x' \in \alpha_i \cap
\beta_j$ has $\eta(x') = \eta(x)$, and hence $\epsilon_{i,j} = \eta(\x) \in
\{+, -\}$. Finally, if $\x = (x_1, \dots, x_g)$ and $\x' = (x'_1, \dots,
x'_g)$, with $x_i \in \alpha_i \cap \beta_{\sigma(i)}$ and $x'_i \in \alpha_i
\cap \beta_{\sigma'(i)}$, then \eqref{eq:grading} and the fact that $\eta(\x) =
\eta(\x')$ imply the final hypothesis.
\end{proof}

To prove Theorem \ref{thm:S3}, we use a simple graph-theoretic argument. Given
a Heegaard diagram $\HH$, let $\Gamma_\HH$ denote the bipartite graph with
vertex sets $\AA = \{A_1, \dots, A_g\}$ and $\BB = \{B_1, \dots, B_g\}$, with
an edge connecting $A_i$ and $B_j$ for each intersection point in $\alpha_i
\cap \beta_j$. The set $\SS_\HH$ thus corresponds to the set of perfect
matchings on $\Gamma_\HH$.

\begin{lemma} \label{lemma:destab}
If $\HH$ is a Heegaard diagram of genus $g>1$, and $\Gamma_\HH$ contains a leaf
(a $1$-valent vertex), then $Y_\HH$ admits a Heegaard diagram $\HH'$ of genus
$g-1$ with a bijection between $\SS_\HH$ and $\SS_{\HH'}$.
\end{lemma}

\begin{proof}
If the vertex $A_i$ is $1$-valent, then the curve $\alpha_i$ intersects one
$\beta$ curve, say $\beta_j$, in a single point and is disjoint from the
remaining $\beta$ curves. By a sequence of handleslides of the $\alpha$ curves,
we may remove any intersections of $\beta_j$ with any $\alpha$ curve other than
$\alpha_i$, without introducing or removing any intersection points. We may
then destabilize to obtain $\HH'$. Since every element of $\SS_\HH$ includes
the unique point of $\alpha_i \cap \beta_j$, we have a bijection between
$\SS_\HH$ and $\SS_{\HH'}$. (Indeed, $\Gamma_\HH'$ is obtained from
$\Gamma_\HH$ by deleting $A_i$ and $B_j$, which does not change the number of
perfect matchings.) The case where $B_i$ is $1$-valent is analogous.
\end{proof}

\begin{proof}[Proof of Theorem \ref{thm:S3}]
Let $\HH$ be a strong Heegaard diagram for $Y$ whose genus $g$ is minimal among
all strong Heegaard diagrams for $Y$. Suppose, toward a contradiction, that
$g>1$. By Lemma \ref{lemma:destab}, $\Gamma_\HH$ has no leaves. By assumption,
$\Gamma_\HH$ has a single perfect matching $\mu$. We direct the edges of
$\Gamma_\HH$ by the following rule: an edge points from $\AA$ to $\BB$ if it is
included in $\mu$ and from $\BB$ to $\AA$ otherwise. Thus, every vertex in
$\AA$ has exactly one outgoing edge, and every vertex in $\BB$ has exactly one
incoming edge. We claim that $\Gamma_\HH$ contains a directed cycle $\sigma$.
Let $\gamma$ be a maximal directed path in $\Gamma_\HH$ that visits each vertex
at most once, and let $v$ be the initial vertex of $\gamma$. If $v \in \BB$,
then there is a unique directed edge $e$ in $\Gamma_\HH$ from some point $w \in
\AA$ to $v$, and $e$ is not included in $\gamma$. Likewise, if $v \in \AA$,
then there is an edge $e$ not in $\gamma$ connecting $v$ and some point $w \in
\BB$ since $v$ is not a leaf, and $e$ is directed from $w$ to $v$ since the
only outgoing edge from $v$ is in $\gamma$. In either case, the maximality of
$\gamma$ implies that $w \in \gamma$, which means that $\gamma \cup e$ contains
a directed cycle. However, $(\mu \minus \sigma) \cup (\sigma \minus \mu)$ is
then another perfect matching for $\Gamma_\HH$.

Thus, the Heegaard diagram $\HH$ is a torus with a single $\alpha$ curve and a
single $\beta$ curve intersecting in a single point, which describes the
standard genus-1 Heegaard splitting of $S^3$.

\end{proof}

\bibliography{thesisBibliography}
\bibliographystyle{plain}

\end{document}